\documentclass{amsart}
\usepackage{amsmath}
\usepackage{amsthm}
\usepackage{graphicx}
\usepackage{booktabs}
\usepackage{multirow}
\usepackage{color}
\usepackage{amsfonts}
\usepackage{amssymb}
\usepackage{mathrsfs}
\usepackage{amscd}
\usepackage{url}

\newcommand{\re}{\mathbb{R}}
\newcommand{\cpx}{\mathbb{C}}

\newcommand{\N}{\mathbb{N}}

\newcommand{\diag}{\mbox{diag}}

\newcommand{\half}{\frac{1}{2}}
\newcommand{\lmd}{\lambda}

\newcommand{\eps}{\epsilon}

\def\af{\alpha}

\def\rank{\mbox{rank}}

\newcommand{\sig}{\sigma}
\newcommand{\Sig}{\Sigma}

\newcommand{\mt}[1]{\mathtt{#1}}

\newcommand{\reff}[1]{(\ref{#1})}

\newcommand{\prm}{\prime}
\newcommand{\mbf}[1]{\mathbf{#1}}

\newcommand{\mc}[1]{\mathcal{#1}}

\newcommand{\supp}[1]{\mbox{supp}(#1)}

\newcommand{\bdes}{\begin{description}}
\newcommand{\edes}{\end{description}}

\newcommand{\bal}{\begin{align}}
\newcommand{\eal}{\end{align}}

\newcommand{\bnum}{\begin{enumerate}}
\newcommand{\enum}{\end{enumerate}}

\newcommand{\bit}{\begin{itemize}}
\newcommand{\eit}{\end{itemize}}

\newcommand{\bea}{\begin{eqnarray}}
\newcommand{\eea}{\end{eqnarray}}
\newcommand{\be}{\begin{equation}}
\newcommand{\ee}{\end{equation}}

\newcommand{\baray}{\begin{array}}
\newcommand{\earay}{\end{array}}

\newcommand{\bsry}{\begin{subarray}}
\newcommand{\esry}{\end{subarray}}

\newcommand{\bca}{\begin{cases}}
\newcommand{\eca}{\end{cases}}

\newcommand{\bcen}{\begin{center}}
\newcommand{\ecen}{\end{center}}

\newcommand{\bbm}{\begin{bmatrix}}
\newcommand{\ebm}{\end{bmatrix}}

\newcommand{\bmx}{\begin{matrix}}
\newcommand{\emx}{\end{matrix}}

\newcommand{\bpm}{\begin{pmatrix}}
\newcommand{\epm}{\end{pmatrix}}

\newcommand{\btab}{\begin{tabular}}
\newcommand{\etab}{\end{tabular}}

\newtheorem{theorem}{Theorem}[section]
\newtheorem{pro}[theorem]{Proposition}
\newtheorem{prop}[theorem]{Proposition}

\newtheorem{ass}[theorem]{Assumption}
\newtheorem{defi}[theorem]{Definition}
\theoremstyle{definition}

\newtheorem{exm}[theorem]{Example}

\numberwithin{equation}{section}

\begin{document}

\title[Tight Relaxations and Lagrange Multiplier Expressions]
{Tight Relaxations for Polynomial Optimization
and Lagrange Multiplier Expressions}

\author{Jiawang Nie}
\address{Department of Mathematics,
University of California at San Diego,
9500 Gilman Drive, La Jolla, CA, USA, 92093.}
\email{njw@math.ucsd.edu}



\begin{abstract}
This paper proposes tight semidefinite relaxations for polynomial optimization.
The optimality conditions are investigated.
We show that generally Lagrange multipliers can be
expressed as polynomial functions in decision variables
over the set of critical points.
The polynomial expressions is determined by linear equations.
Based on these expressions, new Lasserre type semidefinite relaxations
are constructed for solving the polynomial optimization.
We show that the hierarchy of new relaxations has finite convergence,
or equivalently, the new relaxations are tight for a finite relaxation order.
\end{abstract}

\maketitle

\section{Introduction}

A general class of optimization problems is
\be \label{pop:c(x)>=0}
\left\{\baray{rl}
f_{\min} :=  \min & f(x)  \\
 s.t. &  c_i(x) = 0 \, ( i \in \mc{E} ), \\
  &  c_j(x) \geq 0 \, ( j \in \mc{I} ),
\earay \right.
\ee
where $f$ and all $c_i, c_j$ are polynomials
in $x :=(x_1, \ldots, x_n)$, the real decision variable. The $\mc{E}$ and $\mc{I}$
are two disjoint finite index sets of constraining polynomials.
Lasserre's relaxations \cite{Las01} are generally used
for solving \reff{pop:c(x)>=0} globally,
i.e., to find the global minimum value $f_{\min}$ and minimizer(s) if any.
The convergence of Lasserre's relaxations is related
to optimality conditions.

\subsection{Optimality conditions}
\label{ssc:opcd}

A general introduction of optimality conditions
in nonlinear programming can be found in \cite[Section~3.3]{Brks}.
Let $u$ be a local minimizer of \reff{pop:c(x)>=0}.
Denote the index set of active constraints
\be \label{df:J(u)}
J(u) \, := \, \{i \in \mc{E}\cup \mc{I} \, \mid \, c_i(u) = 0\}.
\ee
If the {\it constraint qualification condition (CQC)} holds at $u$,
i.e., the gradients $\nabla c_i(u)$ \, $(i \in J(u))$
are linearly independent ($\nabla$ denotes the gradient), then
there exist Lagrange multipliers
$\lmd_i \, (i \in \mc{E} \cup \mc{I})$ satisfying
\be  \label{opcd:one}
\nabla f(u) =  \sum_{ i \in \mc{E} \cup \mc{I}  }
\lmd_i \nabla c_i(u),
\ee
\be \label{opcd:two}
 \quad
c_i(u) = 0 \, (i \in \mc{E}), \quad
\lmd_j c_j(u) = 0 \, (j \in \mc{I}),
\ee
\be \label{opcd:three}
c_j(u) \geq 0 \, (j \in \mc{I}), \quad
\lmd_j \geq 0 \, (j \in \mc{I}).
\ee
The second equation in \reff{opcd:two}
is called the {\it complementarity condition}.
If $\lmd_j + c_j(u) >0$ for all $j \in \mc{I}$,
the {\it strict complementarity condition} (SCC) is said to hold.
For the $\lmd_i$'s satisfying
\reff{opcd:one}-\reff{opcd:three},
the associated Lagrange function is
\[
\mathscr{L}(x) := f(x) -  \sum_{ i \in \mc{E} \cup \mc{I}  }
\lmd_i c_i(x).
\]
Under the constraint qualification condition,
the {\it second order necessary condition} (SONC) holds at $u$,
i.e., ($\nabla^2$ denotes the Hessian)
\be \label{opt:sonc}
v^T \Big( \nabla^2 \mathscr{L}(u) \Big) v \geq 0 \quad  \mbox{ for all } \,
v \in  \bigcap_{i\in J(u)} \nabla c_i(u)^\perp.
\ee
Here, $\nabla c_i(u)^\perp$ is the orthogonal complement of
$\nabla c_i(u)$.
If it further holds that
\be \label{sosc-optcnd}
v^T \Big( \nabla^2 \mathscr{L}(u) \Big) v > 0 \quad  \mbox{ for all } \,
0 \ne v  \in \bigcap_{i\in J(u)} \nabla c_i(u)^\perp,
\ee
%
%
then the {\it second order sufficient condition} (SOSC) is said to hold.
If the constraint qualification condition holds at $u$,
then \reff{opcd:one}, \reff{opcd:two} and \reff{opt:sonc}
are necessary conditions for $u$ to be a local minimizer.
If \reff{opcd:one}, \reff{opcd:two}, \reff{sosc-optcnd}
and the strict complementarity condition
hold, then $u$ is a strict local minimizer.

\subsection{Some existing work}

Under the archimedean condition (see \S\ref{sc:pre}),
the hierarchy of Lasserre's relaxations converges asymptotically \cite{Las01}.
Moreover, in addition to the archimedeanness,
if the constraint qualification, strict complementarity,
and second order sufficient conditions hold at every global minimizer,
then the Lasserre's hierarchy converges in finitely many steps \cite{Nie-opcd}.
For convex polynomial optimization, the Lasserre's hierarchy
has finite convergence under the strict convexity
or sos-convexity condition \cite{dKlLau11,Las09}.
For unconstrained polynomial optimization,
the standard sum of squares relaxation was proposed in \cite{PS03}.
When the equality constraints
define a finite set, the Lasserre's hierarchy
also has finite convergence, as shown in \cite{LLR08,Lau07,Nie-rVar}.
Recently, a bounded degree hierarchy of relaxations
was proposed for solving polynomial optimization~\cite{LTY15}.
General introductions to polynomial optimization and moment problems
can be found in the books and surveys \cite{LasBok,LBK15,Lau,LauICM,Sch09}.
Lasserre's relaxations provide lower bounds for the minimum value.
There also exist methods that compute upper bounds \cite{dKLLS15,Las11}.
A convergence rate analysis for such upper bounds is given in \cite{dKLS16,dKHL16}.
When a polynomial optimization problem does not have minimizers
(i.e., the infimum is not achievable), there are relaxation methods for computing
the infimum \cite{Schw06,VuiSon}.

A new type of Lasserre relaxations, based on Jacobian representations,
were recently proposed in \cite{Nie-jac}.
The hierarchy of such relaxations always has finite convergence,
when the tuple of constraining polynomials is {\it nonsingular}
(i.e., at every point in $\cpx^n$,
the gradients of active constraining polynomial are linearly independent;
see Definition~\ref{def:nonsig}).
When there are only equality constraints
$c_1(x)=\cdots = c_m(x)=0$, the method needs
the maximal minors of the matrix
\[
\bbm  \nabla f(x)  &  \nabla c_1(x) & \cdots & \nabla c_m(x) \ebm.
\]
When there are inequality constraints,
it requires to enumerate all possibilities of active constraints.
The method in \cite{Nie-jac} is expensive when there are a lot of constraints.
For unconstrained optimization, it is reduced to the gradient
sum of squares relaxations in \cite{NDS}.

\subsection{New contributions}

When Lasserre's relaxations are used to solve polynomial optimization,
the following issues are typically of concerns:
\bit

\item The convergence depends on the archimedean condition (see \S\ref{sc:pre}),
which is satisfied only if the feasible set is compact.
If the set is noncompact, how can we get convergent relaxations?

\item The cost of Lasserre's relaxations depends significantly on the relaxation order.
For a fixed order, can we construct tighter
relaxations than the standard ones?

\item When the convergence of Lasserre's relaxations is slow,
can we construct new relaxations
whose convergence is faster?

\item When the optimality conditions fail to hold,
the Lasserre's hierarchy might not have finite convergence.
Can we construct a new hierarchy of stronger relaxations
that also has finite convergence for such cases?

\eit

This paper addresses the above issues.
We construct tighter relaxations by using optimality conditions.
In \reff{opcd:one}-\reff{opcd:two}, under the constraint qualification condition,
the Lagrange multipliers $\lmd_i$ are uniquely determined by $u$.
%
%
Consider the polynomial system in $(x,\lmd)$:
\be \label{eq:cfx:lmd}
\sum_{ i \in \mc{E} \cup \mc{I}  }
\lmd_i \nabla c_i(x) = \nabla f(x), \,\,
c_i(x) = 0 \, (i \in \mc{E}),  \,\,
\lmd_j c_j(x) = 0 \, (j \in \mc{I}).
\ee
A point $x$ satisfying \reff{eq:cfx:lmd} is called a {\it critical point},
and such $(x,\lmd)$ is called a critical pair.
In \reff{eq:cfx:lmd}, once $x$ is known,
$\lmd$ can be determined by linear equations.
Generally, the value of $x$ is not known.
One can try to express $\lmd$ as a rational function in $x$.
Suppose $\mc{E}\cup \mc{I} = \{1,\ldots, m\}$ and denote
\[
G(x) := \bbm \nabla c_1(x) & \cdots & \nabla c_m(x) \ebm.
\]
When $m\leq n$ and $\rank\,G(x)=m$,
we can get the rational expression
\be \label{lmd:rat:fun}
\lmd = \big( G(x)^TG(x) \big)^{-1} G(x)^T \nabla f(x).
\ee
Typically, the matrix inverse $\big( G(x)^TG(x) \big)^{-1}$
is expensive for usage. The denominator
$\det \big( G(x)^TG(x) \big)$ is typically a high degree polynomial.
When $m>n$,  $G(x)^TG(x)$ is always singular
and we cannot express $\lmd$ as in \reff{lmd:rat:fun}.

Do there exist polynomials $p_i$ ($i \in \mc{E} \cup \mc{I}$)
such that each
\be \label{lmdi=pi}
\lmd_i = p_i(x)
\ee
for all $(x,\lmd)$ satisfying \reff{eq:cfx:lmd}?
If they exist, then we can do:

\bit

\item The polynomial system \reff{eq:cfx:lmd}
can be simplified to
\be \label{kkt:cfx}
\sum_{ i \in \mc{E} \cup \mc{I}  }
p_i(x) \nabla c_i(x) = \nabla f(x), \,\,
c_i(x) =  0 (i \in \mc{E}), \,  \,
p_j(x) c_j(x) = 0  (j \in \mc{I}).
\ee

\item For each $j \in \mc{I}$, the sign condition
$\lmd_j \geq 0$ is equivalent to
\be \label{kkt:lmd>=0}
p_j(x) \geq 0.
\ee

\eit
The new conditions \reff{kkt:cfx} and \reff{kkt:lmd>=0}
are only about the variable $x$, not $\lmd$.
They can be used to construct tighter relaxations
for solving \reff{pop:c(x)>=0}.

When do there exist polynomials $p_i$ satisfying \reff{lmdi=pi}?
If they exist, how can we compute them?
How can we use them to construct tighter relaxations?
Do the new relaxations have advantages over the old ones?
These questions are the main topics of this paper.
Our major results are:

\bit

\item We show that the polynomials $p_i$ satisfying \reff{lmdi=pi}
always exist when the tuple of constraining polynomials is nonsingular
(see Definition~\ref{def:nonsig}). Moreover, they can be
determined by linear equations.

\item Using the new conditions \reff{kkt:cfx}-\reff{kkt:lmd>=0},
we can construct tight relaxations for solving \reff{pop:c(x)>=0}.
To be more precise, we construct a hierarchy of
new relaxations, which has finite convergence.
This is true even if the feasible set is noncompact and/or
the optimality conditions fail to hold.

\item For every relaxation order, the new relaxations
are tighter than the standard ones in the prior work.

\eit

The paper is organized as follows.
Section~\ref{sc:pre} reviews some basics in polynomial optimization.
Section~\ref{sc:tgtrlx} constructs new relaxations
and proves their tightness.
Section~\ref{sc:PH} characterizes when the polynomials
$p_i$'s satisfying \reff{lmdi=pi} exist
and shows how to determine them, for polyhedral constraints.
Section~\ref{sc:gencon} discusses the case of general nonlinear constraints.
Section~\ref{sc:exm} gives examples of using the new relaxations.
Section~\ref{sc:dis} discusses some related issues.

\section{Preliminaries}
\label{sc:pre}

\noindent
{\bf Notation}
The symbol $\N$ (resp., $\re$, $\cpx$) denotes the set of
nonnegative integral (resp., real, complex) numbers.
The symbol $\re[x] := \re[x_1,\ldots,x_n]$
denotes the ring of polynomials in $x:=(x_1,\ldots,x_n)$
with real coefficients.
The $\re[x]_d$ stands for the set of real polynomials
with degrees $\leq d$. Denote
\[
\N^n_d :=\{ \af:= (\af_1, \ldots, \af_n) \in \N^n \mid
|\af|:=\af_1+\cdots+\af_n \leq d \}.
\]
For a polynomial $p$, $\deg(p)$ denotes its total degree.
For $t \in \re$, $\lceil t \rceil$ denotes
the smallest integer $\geq t$.
For an integer $k>0$, denote
$
[k] \, := \, \{1,2,\ldots, k\}.
$
For $x =(x_1,\ldots, x_n)$ and $\af = (\af_1, \ldots, \af_n)$, denote
\[
x^\af \, := \, x_1^{\af_1} \cdots x_n^{\af_n}, \quad
[x]_{d} := \bbm 1 & x_1 &\cdots & x_n & x_1^2 & x_1x_2 & \cdots & x_n^{d}\ebm^T.
\]
The superscript $^T$ denotes the transpose of a matrix/vector.
The $e_i$ denotes the $i$th standard unit vector,
while $e$ denotes the vector of all ones.
The $I_m$ denotes the $m$-by-$m$ identity matrix.
By writing $X\succeq 0$ (resp., $X\succ 0$), we mean that
$X$ is a symmetric positive semidefinite (resp., positive definite) matrix.
For matrices $X_1,\ldots, X_r$, $\diag(X_1, \ldots, X_r)$
denotes the block diagonal matrix whose diagonal blocks
are $X_1,\ldots, X_r$. In particular, for a vector $a$,
$\diag(a)$ denotes the diagonal matrix whose
diagonal vector is $a$. For a function $f$ in $x$,
$f_{x_i}$ denotes its partial derivative with respect to $x_i$.

We review some basics in computational algebra
and polynomial optimization.
They could be found in \cite{CLO97,LasBok,LBK15,Lau,LauICM}.
An ideal $I$ of $\re[x]$ is a subset
such that $ I \cdot \re[x] \subseteq I$
and $I+I \subseteq I$. For a tuple $h := (h_1,\ldots,h_m)$ of polynomials,
$\mbox{Ideal}(h)$ denotes the smallest ideal containing all $h_i$,
which is the set
\[
h_1 \cdot \re[x] + \cdots + h_m  \cdot \re[x].
\]
The $2k$th {\it truncation} of $\mbox{Ideal}(h)$ is the set
\[
\mbox{Ideal}(h)_{2k}  \, := \,
h_1 \cdot \re[x]_{2k-\deg(h_1)} + \cdots + h_m  \cdot \re[x]_{2k-\deg(h_m)}.
\]
The truncation $\mbox{Ideal}(h)_{2k}$ depends on
the generators $h_1,\ldots, h_m$.
%
For an ideal $I$,
its complex and real varieties are respectively defined as
\begin{align*}
\mc{V}_{\cpx}(I) := \{v\in \cpx^n \mid \, p(v) = 0 \, \forall \, p \in I \}, \quad
\mc{V}_{\re}(I)  :=  \mc{V}_{\cpx}(I) \cap \re^n.
\end{align*}

A polynomial $\sig$ is said to be a sum of squares (SOS)
if $\sig = s_1^2+\cdots+ s_k^2$ for some polynomials $s_1,\ldots, s_k \in \re[x]$.
The set of all SOS polynomials in $x$ is denoted as $\Sig[x]$.
For a degree $d$, denote the truncation
\[
\Sig[x]_d := \Sig[x] \cap \re[x]_d.
\]
For a tuple $g=(g_1,\ldots,g_t)$,
its {\it quadratic module} is the set
\[
\mbox{Qmod}(g):=  \Sig[x] + g_1 \cdot \Sig[x] + \cdots + g_t \cdot \Sig[x].
\]
The $2k$th truncation of $\mbox{Qmod}(g)$ is the set
\[
\mbox{Qmod}(g)_{2k} \, := \,
\Sig[x]_{2k} + g_1 \cdot \Sig[x]_{2k - \deg(g_1)}
+ \cdots + g_t \cdot \Sig[x]_{2k - \deg(g_t)}.
\]
The truncation $\mbox{Qmod}(g)_{2k}$ depends on
the generators $g_1,\ldots, g_t$. Denote
\be \label{df:IQ(h,g)}
\left\{ \baray{lcl}
\mbox{IQ}(h,g) &:=& \mbox{Ideal}(h) + \mbox{Qmod}(g), \\
\mbox{IQ}(h,g)_{2k} &:=& \mbox{Ideal}(h)_{2k} + \mbox{Qmod}(g)_{2k}.
\earay \right.
\ee
The set $\mbox{IQ}(h,g)$ is said to be {\it archimedean}
if there exists $p \in \mbox{IQ}(h,g)$ such that
$p(x) \geq 0$ defines a compact set in $\re^n$.
If $\mbox{IQ}(h,g)$ is archimedean, then
\[
K \, := \, \{x \in \re^n \mid h(x)=0, \, g(x)\geq 0 \}
\]
must be a compact set. Conversely, if $K$ is compact, say, $K \subseteq B(0,R)$
(the ball centered at $0$ with radius $R$), then
$\mbox{IQ}(h, (g, R^2-x^Tx) )$ is always archimedean
and $h=0, \, (g, R^2-x^Tx) \geq 0$
give the same set $K$.

\begin{theorem}[Putinar \cite{Put}]\label{thm:PutThm}
Let $h,g$ be tuples of polynomials in $\re[x]$.
Let $K$ be as above. Assume $\mbox{IQ}(h,g)$ is archimedean.
If a polynomial $f\in \re[x]$ is positive on $K$,
then $f\in \mbox{IQ}(h,g)$.
\end{theorem}

Interestingly, if $f$ is only nonnegative on $K$ but standard optimality
conditions hold (see Subsection~\ref{ssc:opcd}), then we still have
$f \in \mbox{IQ}(h,g)$ \cite{Nie-opcd}.

Let $\re^{\N_d^n}$ be the space of real multi-sequences indexed by
$\af \in \N^n_d$. A vector in $\re^{\N_d^n}$ is called a
{\it truncated multi-sequence (tms)} of degree $d$.
A tms $y :=(y_\af)_{ \af \in \N_d^n}$ gives the Riesz functional
$\mathscr{R}_y$ acting on $\re[x]_d$ as
{\small
\be
\mathscr{R}_y\Big(\sum_{\af \in \N_d^n}
f_\af x^\af  \Big) := \sum_{\af \in \N_d^n}  f_\af y_\af.
\ee
} \noindent
For $f \in \re[x]_d$ and $y \in \re^{\N_d^n}$, we denote
\be \label{df:<p,y>}
\langle f, y \rangle := \mathscr{R}_y(f).
\ee
Let $q \in \re[x]_{2k}$.
The $k$th {\it localizing matrix} of $q$,
generated by $y \in \re^{\N^n_{2k}}$,
is the symmetric matrix $L_q^{(k)}(y)$ such that
\be  \label{LocMat}
vec(a_1)^T \Big( L_q^{(k)}(y) \Big) vec(a_2)  = \mathscr{R}_y(q a_1 a_2)
\ee
for all $a_1,a_2 \in \re[x]_{k - \lceil \deg(q)/2 \rceil}$.
(The $vec(a_i)$ denotes the coefficient vector of $a_i$.)
When $q = 1$, 
$L_q^{(k)}(y)$ is called a {\it moment matrix} and we denote
\be \label{MomMat}
M_k(y):= L_{1}^{(k)}(y).
\ee
The columns and rows of $L_q^{(k)}(y)$, as well as $M_k(y)$,
are indexed by $\af \in \N^n$ with $2|\af| + \deg(q) \leq 2k$.
When $q = (q_1, \ldots, q_r)$ is a tuple of polynomials, we define
\be  \label{block:LM}
L_q^{(k)}(y) \, := \, \mbox{diag}\Big(
L_{q_1}^{(k)}(y), \ldots,  L_{q_r}^{(k)}(y)
\Big),
\ee
a block diagonal matrix.
For the polynomial tuples $h,g$ as above, the set
\be
\mathscr{S}(h,g)_{2k} :=
\Big \{
\left. y \in \re^{ \N_n^{2k} } \right|
L_h^{(k)}(y) =0, \,  L_g^{(k)}(y) \succeq 0
\Big \}
\ee
is a spectrahedral cone in $\re^{ \N_n^{2k} }$.
The set $\mbox{IQ}(h,g)_{2k}$
is also a convex cone in $\re[x]_{2k}$.
The dual cone of $\mbox{IQ}(h,g)_{2k}$
is precisely $\mathscr{S}(h,g)_{2k}$ \cite{LBK15,Lau,Nie-LinOpt}.
This is because $\langle p, y \rangle \geq 0$
for all $p \in  \mbox{IQ}(h,g)_{2k}$
and for all $y \in \mathscr{S}(h,g)_{2k}$.

\section{The construction of tight relaxations}
\label{sc:tgtrlx}

Consider the polynomial optimization problem \reff{pop:c(x)>=0}.
Let
\[
\lmd :=(\lmd_i)_{ i \in \mc{E}\cup\mc{I} }
\]
be the vector of Lagrange multipliers.
Denote the set
\be \label{kkt:ffifj}
\mc{K} \, := \, \left\{
(x, \lmd) \in \re^n  \times \re^{ \mc{E} \cup  \mc{I} }
\left| \baray{c}
c_i(x) = 0 (i \in \mc{E}), \,  \lmd_j c_j(x) = 0\, (j \in \mc{I})   \\
\nabla f(x) = \sum \limits_{ i \in \mc{E} \cup \mc{I} }
\lmd_i \nabla c_i(x)
\earay \right.
\right \}.
\ee
Each point in $\mc{K}$ is called a critical pair.
The projection
\be  \label{set:Kc}
\mc{K}_c :=  \{ u \mid (u, \lmd) \in \mc{K} \}
\ee
is the set of all real critical points.
To construct tight relaxations for solving \reff{pop:c(x)>=0},
we need the following assumption for Lagrange multipliers.

\begin{ass} \label{ass:lmd:p}
For each $i \in \mc{E} \cup \mc{I}$,
there exists a polynomial $p_i \in \re[x]$ such that for all
$(x,\lmd) \in \mc{K}$  it holds that
\[
\lmd_i \, =  \, p_i(x).
\]
\end{ass}

Assumption~\ref{ass:lmd:p} is generically satisfied,
as shown in Proposition~\ref{pro:c-ns-gen}.
For the following special cases, we can get polynomials
$p_i$ explicitly.
\bit

\item (Simplex) For the simplex
$\{e^Tx-1 =0, \, x_1 \geq 0, \ldots, x_n \geq 0\}$,
it corresponds to that $\mc{E} = \{0\}$, $\mc{I} = [n]$,
$
c_0(x)=e^Tx-1, \, c_j(x)=x_j  \, (j\in [n]).
$
The Lagrange multipliers can be expressed as
\be
\lmd_0 = x^T \nabla f(x), \quad \lmd_j = f_{x_j}-x^T\nabla f(x) \quad (j \in [n]).
\ee

\item (Hypercube) For the hypercube $[-1,1]^n$,
it corresponds to that $\mc{E} = \emptyset$, $\mc{I} = [n]$ and
each $c_j(x)= 1-x_j^2$. We can show that
\be
\lmd_j = - \half x_j f_{x_j} \quad (j\in [n]).
\ee

\item (Ball or sphere)
The constraint is $1-x^Tx=0$ or $1-x^Tx\geq 0$.
It corresponds to that $\mc{E} \cup \mc{I} = \{1\}$ and
$c_1 = 1-x^Tx$. We have
\be
\lmd_1 = -\half x^T \nabla f(x).
\ee

%
%

\item (Triangular constraints)
Suppose $\mc{E} \cup \mc{I} =\{1, \ldots, m\} $ and each
\[
c_i(x) \, =  \tau_i x_i + q_i(x_{i+1}, \ldots, x_n)
\]
for some polynomials $q_i \in \re[x_{i+1}, \ldots, x_n]$
and scalars $\tau_i \ne 0$.
The matrix $T(x)$, consisting of the first $m$ rows of
$
[ \nabla c_1(x), \ldots, \nabla c_m(x) ],
$
is an invertible lower triangular matrix with constant diagonal entries.
Then,
\[
\lmd = T(x)^{-1} \cdot
\bbm f_{x_1} & \cdots & f_{x_m} \ebm^T.
\]
Note that the inverse $T(x)^{-1}$ is a matrix polynomial.

\eit
\noindent
For more general constraints, we can also express $\lmd$
as a polynomial function in $x$ on the set $\mc{K}_c$.
This will be discussed in \S\ref{sc:PH} and \S\ref{sc:gencon}.

%
%
%

For the polynomials $p_i$ as in Assumption~\ref{ass:lmd:p}, denote
\be \label{df:phi:psi}
\phi \, :=  \Big( \nabla f - \sum \limits_{ i \in \mc{E} \cup \mc{I} }
p_i \nabla c_i, \,\, \big( p_j c_j \big)_{ j \in \mc{I} } \Big), \quad
\psi \, :=  \big(  p_j \big)_{ j \in \mc{I} }.
\ee
When the minimum value $f_{\min}$ of \reff{pop:c(x)>=0}
is achieved at a critical point,
\reff{pop:c(x)>=0} is equivalent to the problem
\be \label{g:kkt:opt}
\left\{\baray{rl}
f_c :=  \min & f(x)  \\
 s.t. &  c_{eq}(x) = 0, \, c_{in}(x) \geq 0, \\
      &  \phi(x) = 0, \, \psi(x) \geq 0.
\earay \right.
\ee
We apply Lasserre relaxations to solve it.
For an integer $k>0$ (called the {\it relaxation order}),
the $k$th order Lasserre's relaxation for \reff{g:kkt:opt} is
\be \label{momrlx:k}
\left\{\baray{rl}
f_k^\prm  := \min & \langle f, y \rangle \\
 s.t. & \langle 1, y \rangle = 1,  M_k(y) \succeq 0 \\
      &  L_{c_{eq}}^{(k)}(y) = 0,   \,  L_{c_{in}}^{(k)}(y) \succeq 0, \\
      &  L_{\phi}^{(k)}(y) = 0,   \,  L_{\psi}^{(k)}(y) \succeq 0, \,
        y \in \re^{ \N_{2k}^n }.
\earay \right.
\ee
Since $x^0 = 1$ (the constant one polynomial),
the condition $\langle 1, y \rangle = 1$ means that $(y)_0=1$.
The dual optimization problem of \reff{momrlx:k} is
\be \label{rlxsos:k}
\left\{\baray{rl}
f_k := \, \max & \gamma  \\
 s.t. &  f-\gamma \in \mbox{IQ}(c_{eq},c_{in})_{2k} +
             \mbox{IQ}(\phi,\psi)_{2k}.
\earay \right.
\ee
We refer to \S\ref{sc:pre} for the notation
used in \reff{momrlx:k}-\reff{rlxsos:k}.
They are equivalent to semidefinite programs (SDPs),
so they can be solved by SDP solvers (e.g., {\tt SeDuMi} \cite{Sturm}).
For $k=1,2,\cdots$, we get a hierarchy of Lasserre relaxations.
In \reff{momrlx:k}-\reff{rlxsos:k}, if we remove the usage of $\phi$ and $\psi$,
they are reduced to standard Lasserre relaxations in \cite{Las01}.
So, \reff{momrlx:k}-\reff{rlxsos:k} are stronger relaxations.

By the construction of $\phi$ as in \reff{df:phi:psi},
Assumption~\ref{ass:lmd:p} implies that
\[
\mc{K}_c = \{ u \in \re^n : \, c_{eq}(u)=0, \, \phi(u)=0 \}.
\]
By Lemma~3.3 of \cite{DNP},
$f$ achieves only finitely many values on $\mc{K}_c$, say,
\be \label{v1<vN}
v_1 < \cdots < v_N.
\ee
A point $u \in \mc{K}_c$ might not be feasible for \reff{g:kkt:opt},
i.e., it is possible that $c_{in}(u) \not\geq 0$ or $\psi(u) \not\geq 0$.
In applications, we are often interested
in the optimal value $f_c$ of \reff{g:kkt:opt}.
When \reff{g:kkt:opt} is infeasible, by convention, we set
\[
f_c \, = \, +\infty.
\]

When the optimal value $f_{\min}$ of \reff{pop:c(x)>=0}
is achieved at a critical point,
$f_c = f_{\min}$. This is the case if the feasible set is compact,
or if $f$ is coercive (i.e., for each $\ell$,
the sublevel set $\{f(x) \leq \ell \}$ is compact),
and the constraint qualification condition holds.
As in \cite{Las01}, one can show that
\be \label{fk<=fc}
f_k \leq f_k^\prm \leq f_c
\ee
for all $k$. Moreover, $\{ f_k \}$ and $\{ f_k^\prm \}$
are both monotonically increasing.
If for some order $k$ it occurs that
\[ f_k = f_k^\prm = f_c,\]
then the $k$th order Lasserre's relaxation
is said to be {\it tight} (or {\it exact}).

\subsection{Tightness of the relaxations}

Let $c_{in},\psi, \mc{K}_c,f_c$ be as above.
We refer to \S\ref{sc:pre} for the notation $\mbox{Qmod}(c_{in},\psi)$.
We begin with a general assumption.

\begin{ass} \label{ass:sig}
There exists $\rho \in \mbox{Qmod}(c_{in},\psi)$
such that if $u\in \mc{K}_c$ and $f(u) < f_c$, then $\rho(u) < 0$.
\end{ass}

In Assumption~\ref{ass:sig}, the hypersurface $\rho(x) = 0$
separates feasible and infeasible critical points.
Clearly, if $u\in \mc{K}_c$ is a feasible point for \reff{g:kkt:opt},
then $c_{in}(u) \geq 0$ and $\psi(u) \geq 0$,
and hence $\rho(u) \geq 0$.
Assumption~\ref{ass:sig} generally holds.
For instance, it is satisfied for the following general cases.

\bit

\item [a)] When there are no inequality constraints,
$c_{in}$ and $\psi$ are empty tuples. Then,
$\mbox{Qmod}(c_{in},\psi) = \Sig[x]$ and
Assumption~\ref{ass:sig} is satisfied for $\rho =0$.

\item [b)] Suppose the set $\mc{K}_c$ is finite, say,
$\mc{K}_c = \{ u_1, \ldots, u_D \}$, and
\[
f(u_1), \ldots, f(u_{t-1}) < f_c \leq
f(u_{t}), \ldots, f(u_D).
\]
Let $\ell_1,\ldots, \ell_D$ be real interpolating polynomials
such that $\ell_i(u_j) = 1$ for $i=j$
and $\ell_i(u_j) = 0$ for $ i \ne j$.
For each $i = 1, \ldots, t$, there must exist $j_i \in \mc{I}$
such that $c_{j_i}(u_i) < 0$.
Then, the polynomial
\be \label{cho:rho:1}
\rho := \sum_{i < t} \frac{-1}{c_{j_i}(u_i)} c_{j_i}(x) \ell_i(x)^2 +
 \sum_{i \geq t}  \ell_i(x)^2
\ee
satisfies Assumption~\ref{ass:sig}.

\item [c)] For each $x$ with $f(x)=v_i<f_c$, at least one of the constraints
$c_j(x)\geq 0, p_j(x) \geq 0 (j \in \mc{I})$ is violated.
Suppose for each critical value $v_i < f_c$,
there exists $g_i \in  \{ c_j, p_j \}_{j \in \mc{I} }$ such that
\[
 g_i < 0 \quad \mbox{ on }  \quad \mc{K}_c \cap \{ f(x) = v_i \}.
\]
Let $\varphi_1, \ldots, \varphi_N$
be real univariate polynomials such that
$\varphi_i( v_j) = 0$ for $i \ne j$ and $\varphi_i( v_j) = 1$ for $i = j$.
Suppose $v_t = f_c$. Then, the polynomial
\be  \label{cho:rho:2}
\rho := \sum_{ i < t } g_i(x) \big( \varphi_i(f(x)) \big)^2 +
\sum_{ i \geq t }  \big( \varphi_i(f(x)) \big)^2
\ee
satisfies Assumption~\ref{ass:sig}.

\eit

We refer to \S\ref{sc:pre}
for the archimedean condition
and the notation $\mbox{IQ}(h,g)$ as in \reff{df:IQ(h,g)}.
The following is about the convergence of relaxations
\reff{momrlx:k}-\reff{rlxsos:k}.

\begin{theorem}  \label{thm:pf:tight}
Suppose $\mc{K}_c \ne \emptyset$ and Assumption~\ref{ass:lmd:p} holds.
If
\bit
\item [i)] $\mbox{IQ}(c_{eq},c_{in}) + \mbox{IQ}(\phi, \psi)$
is archimedean, {\bf or}

\item [ii)] $\mbox{IQ}(c_{eq},c_{in}) $ is archimedean, {\bf or}

\item [iii)] Assumption~\ref{ass:sig} holds,

\eit
then $f_k  = f_k^\prm  = f_c$ for all $k$ sufficiently large.
Therefore, if the minimum value $f_{\min}$ of \reff{pop:c(x)>=0}
is achieved at a critical point,
then $f_k  = f_k^\prm  = f_{\min}$ for all $k$ big enough
if one of the conditions i)-iii) is satisfied.
\end{theorem}

\noindent
{\it Remark:} In Theorem~\ref{thm:pf:tight},
the conclusion holds if anyone of conditions i)-iii) is satisfied.
The condition ii) is only about constraining polynomials of \reff{pop:c(x)>=0}.
It can be checked without $\phi,\psi$.
Clearly, the condition ii) implies the condition i).

The proof for Theorem~\ref{thm:pf:tight} is given in the following.
The main idea is to consider the set of critical points.
It can be expressed as a union of subvarieties.
The objective $f$ is a constant in each one of them.
We can get an SOS type representation for $f$
on each subvariety, and then construct a single one for
$f$ over the entire set of critical points.

\begin{proof}[Proof of Theorem~\ref{thm:pf:tight}]
Clearly, every point in the complex variety
\[
\mc{K}_1  :=  \{ x \in \cpx^n \mid  c_{eq}(x) = 0, \phi(x) = 0  \}
\]
is a critical point. By Lemma~3.3 of \cite{DNP},
the objective $f$ achieves finitely many real values on
$\mc{K}_c = \mc{K}_1 \cap \re^n$,
say, they are $v_1 < \cdots < v_N$.
Up to the shifting of a constant in $f$, we can further assume that
$
f_c = 0.
$
Clearly, $f_c$ equals one of $v_1, \ldots, v_N$, say
$
v_t = f_c = 0.
$

\bigskip
\noindent
{\bf Case I:}\, Assume 
$\mbox{IQ}(c_{eq},c_{in}) + \mbox{IQ}(\phi, \psi)$
is archimedean. Let
\[ I \, := \, \mbox{Ideal}(c_{eq}, \phi), \]
the critical ideal. Note that $\mc{K}_1 = \mc{V}_{\cpx}(I)$.
The variety $\mc{V}_{\cpx}(I)$ is a union of irreducible subvarieties,
say, $V_1, \ldots, V_\ell$.
If $V_i \cap \re^n \ne \emptyset$, then $f$ is a real constant on $V_i$,
which equals one of $v_1, \ldots, v_N$.
This can be implied by Lemma~3.3 of \cite{DNP} and Lemma~3.2 of \cite{Nie-jac}.
Denote the subvarieties of $\mc{V}_{\cpx}(I)$:
\[
T_i  := \mc{K}_1 \cap \{ f(x) = v_i \} \quad (i=t, \ldots,N).
\]
Let $T_{t-1}$ be the union of irreducible subvarieties
$V_i$, such that either $V_i \cap \re^n = \emptyset$
or $f \equiv v_j$ on $V_i$ with $v_j < v_t = f_c$.
%
%
Then, it holds that
\[
\mc{V}_{\cpx}(I) = T_{t-1} \cup T_t \cup \cdots \cup T_N.
\]
By the primary decomposition of $I$ \cite{GPsig,Stu02}, there exist ideals
$I_{t-1}, I_t, \ldots, I_N \subseteq \re[x]$ such that
\[
I = I_{t-1} \,\cap \, I_t \,\cap \,\cdots \,\cap \, I_N
\]
and $T_i = \mc{V}_{\cpx}(I_i)$ for all $i = t-1, t, \ldots, N$.
Denote the semialgebraic set
\be \label{set:S}
S := \{ x \in \re^n \mid \, c_{in}(x) \geq 0, \, \psi(x) \geq 0 \}.
\ee

\smallskip
For $i=t-1$, we have $\mc{V}_{\re}(I_{t-1}) \cap S = \emptyset$,
because $v_1, \ldots, v_{t-1} < f_c$.
By the Positivstellensatz \cite[Corollary~4.4.3]{BCR},
there exists $p_0 \in \mbox{Preord}(c_{in},\psi)$\footnote
{It is the preordering of the polynomial tuple
$(c_{in},\psi)$; see \S\ref{ssc:prod}.}
satisfying
$2 + p_0 \in I_{t-1}$.
Note that $1+p_0 >0$ on $\mc{V}_{\re}(I_{t-1}) \cap S$.
The set $I_{t-1} + \mbox{Qmod}(c_{in}, \psi)$ is archimedean,
because $I \subseteq I_{t-1}$ and
\[
\mbox{IQ}(c_{eq},c_{in}) + \mbox{IQ}(\phi,\psi)
\, \subseteq \,
I_{t-1} + \mbox{Qmod}(c_{in}, \psi).
\]
By Theorem~\ref{thm:PutThm}, we have
\[
p_1 := 1+ p_0 \in I_{t-1} + \mbox{Qmod}(c_{in}, \psi).
\]
Then, $1+p_1   \in I_{t-1}$.
There exists $p_2 \in \mbox{Qmod}(c_{in}, \psi)$ such that
\[
-1 \equiv p_1 \equiv  p_2  \, \mod \, I_{t-1}.
\]
Since $f = (f/4+1)^2 -1 \cdot (f/4-1)^2$, we have
\begin{align*}
f & \equiv  \sig_{t-1} :=  \left\{
(f/4+1)^2 + p_2 (f/4-1)^2  \right\}
\quad \mod \quad I_{t-1}.
\end{align*}
So, when $k$ is big enough, we have
$\sig_{t-1} \in \mbox{Qmod}(c_{in}, \psi)_{2k}$.

\smallskip
For $i=t$, $v_t=0$ and $f(x)$ vanishes on $\mc{V}_{\cpx}(I_t)$.
By Hilbert's Strong Nullstellensatz \cite{CLO97},
there exists an integer $m_t > 0$
such that $f^{m_t} \in I_t$.
Define the polynomial
\[
s_t(\eps) :=  \sqrt{\eps}  {\sum}_{j=0}^{m_t-1}
\binom{1/2}{j} \eps^{-j} f^j.
\]
Then, we have that
\[
D_1 := \eps (1 + \eps^{-1} f )  - \big( s_t(\eps) \big)^2
 \equiv  0 \,\,\,\, \mod\,\,\,\, I_t\, .
\]
This is because in the subtraction of $D_1$,
after expanding $\big( s_t(\eps) \big)^2$,
all the terms $f^j$ with $j < m_t$ are cancelled
and $f^j \in I_t$ for $j \geq m_t$.
So, $D_1 \in I_t$. Let $\sig_t(\eps) := s_t(\eps)^2$, then
$f + \eps - \sig_t(\eps) = D_1$ and
\be \label{f+eps=bj(eps)f}
f + \eps - \sig_t(\eps)  \, = \,
{\sum}_{j=0}^{m_t-2}  b_j(\eps) f^{m_t+j}
\ee
for some real scalars $b_j(\eps)$, depending on $\eps$.

\smallskip
For each $i = t+1, \ldots, N$, $v_i>0$ and $f(x)/v_i - 1$
vanishes on $\mc{V}_{\cpx}(I_i)$.
By Hilbert's Strong Nullstellensatz \cite{CLO97},
there exists $0< m_i \in \N$ such that
$(f/v_i - 1)^{m_i} \in I_i.$
Let
\[
s_i   := \sqrt{v_i}
{\sum}_{j=0}^{m_i-1} \binom{1/2}{j} (f/v_i - 1)^j.
\]
Like for the case $i=t$, we can similarly show that
$f - s_i^2 \, \in \, I_i$.
Let $\sig_i =  s_i^2$, then
$f - \sig_i \, \in \, I_i$.

\smallskip
Note that $\mc{V}_{\cpx}(I_i) \cap \mc{V}_{\cpx}(I_j) = \emptyset$
for all $i \ne j$. By Lemma~3.3 of \cite{Nie-jac}, there exist polynomials
$a_{t-1}, \ldots, a_N \in \re[x]$ such that
\[
a_{t-1}^2+\cdots+a_N^2-1 \, \in \, I, \quad a_i \in
\bigcap_{ i \ne j \in \{ t-1, \ldots,N \} } I_j.
\]
For $\eps >0$, denote the polynomial
\[
\sig_\eps \, := \, \sig_t(\eps) a_t^2 + \sum_{ t \ne j \in \{ t-1, \ldots,N \} }
(\sig_j+\eps) a_j^2,
\]
then
\[
\baray{rcl}
 f + \eps  - \sig_\eps
 & = &   (f+\eps)(1-a_{t-1}^2-\cdots-a_N^2) + \\
 &  & \sum_{ t \ne i \in \{ t-1, \ldots,N \} }
         (f-\sig_i) a_i^2 + (f+\eps-\sig_t(\eps) ) a_t^2.
\earay
\]
For each $i \ne t$, $f-\sig_i \in I_i$, so
\[
(f-\sig_i) a_i^2 \, \in \, \bigcap_{j=t-1}^N I_j = I.
\]
Hence, there exists $k_1 >0$ such that
\[
(f-\sig_i) a_i^2 \, \in \, I_{2k_1} \, \,
(t \ne i \in \{t-1, \ldots, N \}).
\]
Since $f+\eps-\sig_t(\eps) \in I_t$, we also have
\[
(f+\eps-\sig_t(\eps) ) a_t^2 \, \in \,  \bigcap_{j=t-1}^N I_j = I.
\]
Moreover, by the equation~\reff{f+eps=bj(eps)f},
\[
(f+\eps-\sig_t(\eps) ) a_t^2  =  \sum_{j=0}^{m_t-2} b_j(\eps) f^{m_t+j} a_t^2.
\]
Each $f^{m_t+j} a_t^2 \in I$, since $f^{m_t+j} \in I_t$.
So, there exists $k_2 >0$ such that for all $\eps >0$
\[
(f+\eps-\sig_t(\eps) ) a_t^2 \, \in \, I_{2k_2}.
\]
Since $1-a_{t-1}^2-\cdots-a_N^2 \in I$,
there also exists $k_3 >0$ such that for all $\eps >0$
\[
(f+\eps)(1-a_{t-1}^2-\cdots-a_N^2) \, \in \, I_{2k_3}.
\]
Hence, if $k^* \geq \max\{k_1, k_2, k_3\}$, then we have
\[
f(x) + \eps  - \sig_\eps \, \in \, I_{2k^*}
\]
for all $\eps >0$. By the construction,
the degrees of all $\sig_i$ and $a_i$ are independent of $\eps$.
So, $\sig_\eps  \, \in \, \mbox{Qmod}(c_{in}, \psi)_{2k^*}$
for all $\eps >0$ if $k^*$ is big enough. Note that
\[
I_{2k^*} + \mbox{Qmod}(c_{in}, \psi)_{2k^*} =
\mbox{IQ}(c_{eq}, c_{in})_{2k^*} +
\mbox{IQ}(\phi, \psi)_{2k^*}.
\]
This implies that
$
f_{k^*} \geq  f_c -\eps
$
for all $\eps >0.$ On the other hand, we always have
$f_{k^*} \leq f_c$. So, $f_{k^*} = f_c$. Moreover,
since $\{ f_k \}$ is monotonically increasing,
we must have $f_{k} = f_c$ for all $k\geq k^*$.

\bigskip
\noindent
{\bf Case II:}\, Assume
$\mbox{IQ}(c_{eq},c_{in})$ is archimedean. Because
\[
\mbox{IQ}(c_{eq},c_{in}) \subseteq \mbox{IQ}(c_{eq},c_{in})
+ \mbox{IQ}(\phi,\psi),
\]
the set $\mbox{IQ}(c_{eq},c_{in})+ \mbox{IQ}(\phi,\psi)$
is also archimedean. Therefore, the conclusion is also true
by applying the result for {\bf Case I}.

\bigskip
\noindent
{\bf Case III:}\, Suppose the Assumption~\ref{ass:sig} holds.
Let $\varphi_1, \ldots, \varphi_N$
be real univariate polynomials such that
$\varphi_i( v_j) = 0$ for $i \ne j$ and $\varphi_i( v_j) = 1$ for $i = j$.
Let
\[
s := s_t + \cdots + s_N \quad \mbox{where each} \quad
s_i :=  ( v_i -  f_c) \big( \varphi_i (  f )  \big)^2.
\]
Then, $s \in \Sig[x]_{2k_4}$ for some integer $k_4 > 0$. Let
\[
\hat{f}:=f - f_c - s.
\]
We show that there exist an integer $\ell > 0$
and $q \in \mbox{Qmod}(c_{in}, \psi)$ such that
\[
\hat{f}^{2\ell} + q  \in \mbox{Ideal}(c_{eq},\phi).
\]
This is because, by Assumption~\ref{ass:sig},
$\hat{f}(x) \equiv 0$ on the set
\[
\mc{K}_2 \, := \, \{ x \in \re^n: \,
c_{eq}(x) = 0, \, \phi(x) = 0, \, \rho(x) \geq 0 \}.
\]
It has only a single inequality.
By the Positivstellensatz~\cite[Corollary~4.4.3]{BCR},
there exist $0< \ell \in \N$ and
$q = b_0 + \rho b_1 $ ($b_0, b_1 \in \Sig[x]$) such that
$
\hat{f}^{2\ell} + q  \in \mbox{Ideal}(c_{eq},\phi).
$
By Assumption~\ref{ass:sig}, $\rho \in \mbox{Qmod}(c_{in}, \psi)$,
so we have $q \in \mbox{Qmod}(c_{in}, \psi)$.

For all $\eps >0$ and $\tau >0$, we have
$
\hat{f} + \eps = \phi_\eps + \theta_\eps
$
where
\[
\phi_\eps = -\tau \eps^{1-2\ell}
\big(\hat{f}^{2\ell} + q \big),
\]
\[
\theta_\eps = \eps \Big(1 + \hat{f}/\eps +
\tau ( \hat{f}/\eps)^{2\ell} \Big)
+ \tau \eps^{1-2\ell} q.
\]
By Lemma~2.1 of \cite{Nie-rVar}, when $\tau \geq \frac{1}{2\ell}$,
there exists $k_5$ such that, for all $\eps >0$,
\[
\phi_\eps \in \mbox{Ideal}(c_{eq},\phi)_{2k_5}, \quad
\theta_\eps \in \mbox{Qmod}(c_{in},\psi)_{2k_5}.
\]
Hence, we can get
\[
f - (f_c -\eps) = \phi_\eps + \sig_\eps,
\]
where $\sig_\eps = \theta_\eps + s \in \mbox{Qmod}(c_{in},\psi)_{2k_5}$
for all $\eps >0$. Note that
\[
\mbox{IQ}(c_{eq},c_{in})_{2k_5} + \mbox{IQ}(\phi,\psi)_{2k_5} =
\mbox{Ideal}(c_{eq},\phi)_{2k_5} + \mbox{Qmod}(c_{in},\psi)_{2k_5}.
\]
For all $\eps>0$,
$\gamma = f_c-\eps$ is feasible in \reff{rlxsos:k} for the order $k_5$,
so $f_{k_5} \geq f_c$.
Because of \reff{fk<=fc} and the monotonicity of $\{ f_k \}$,
we have $f_{k} = f_{k}^\prm = f_c$ for all $k \geq k_5$.
\end{proof}

\subsection{Detecting tightness and extracting minimizers}

The optimal value of \reff{g:kkt:opt} is $f_c$,
and the optimal value of \reff{pop:c(x)>=0} is $f_{\min}$.
If $f_{\min}$ is achievable at a critical point, then $f_c  = f_{\min}$.
In Theorem~\ref{thm:pf:tight}, we have shown that
$f_k = f_c$ for all $k$ big enough,
where $f_k$ is the optimal value of \reff{rlxsos:k}.
The value $f_c$ or $f_{\min}$ is often not known.
How do we detect the tightness
$
f_k =  f_c
$
in computation? The flat extension or
flat truncation condition~\cite{CuFi05,HenLas05,Nie-ft}
can be used for checking tightness.
Suppose $y^*$ is a minimizer of \reff{momrlx:k}
for the order $k$. Let
\be \label{deg:d}
 d \, := \,
\lceil  \deg(c_{eq}, c_{in}, \phi, \psi)/2 \rceil.
\ee
If there exists an integer $t \in [d, k]$ such that
\be \label{ft:Mty}
\rank \, M_t (y^*) \,= \, \rank \, M_{t-d}(y^*)
\ee
then $f_k  = f_{c}$ and we can get $r := \rank \, M_t (y^*)$
minimizers for \reff{g:kkt:opt}~\cite{CuFi05,HenLas05,Nie-ft}.
The method in \cite{HenLas05}
can be used to extract minimizers.
It was implemented in the software
{\tt GloptiPoly~3} \cite{GloPol3}.
Generally, \reff{ft:Mty}
can serve as a sufficient and necessary condition
for detecting tightness.
The case that \reff{g:kkt:opt} is infeasible
(i.e., no critical points satisfy the constraints
$c_{in} \geq 0, \psi \geq 0$)
can also be detected by solving the relaxations
\reff{momrlx:k}-\reff{rlxsos:k}.

\begin{theorem} \label{thm:ft}
Under Assumption~\ref{ass:lmd:p},
the relaxations \reff{momrlx:k}-\reff{rlxsos:k}
have the following properties:

\bit

\item [i)] If \reff{momrlx:k} is infeasible for some order $k$,
then no critical points satisfy the constraints
$c_{in} \geq 0, \psi \geq 0$,
i.e., \reff{g:kkt:opt} is infeasible.

\item [ii)] Suppose Assumption~\ref{ass:sig} holds.
If \reff{g:kkt:opt} is infeasible,
then the relaxation \reff{momrlx:k}
must be infeasible when the order $k$ is big enough.

\eit
\noindent
In the following, assume \reff{g:kkt:opt} is feasible
(i.e., $f_c < + \infty$). Then, for all $k$ big enough,
\reff{momrlx:k} has a minimizer $y^*$. Moreover,

\bit

\item [iii)] If \reff{ft:Mty} is satisfied for some $t \in [d,k]$,
then $f_k = f_c$.

\item [iv)] If Assumption~\ref{ass:sig} holds
and \reff{g:kkt:opt} has finitely many minimizers, then
every minimizer $y^*$ of \reff{momrlx:k} must satisfy \reff{ft:Mty}
for some $t \in [d,k]$, when $k$ is big enough.

\eit

\end{theorem}
\begin{proof}
By Assumption~\ref{ass:lmd:p},
$u$ is a critical point if and only if
$c_{eq}(u) = 0, \phi(u)=0$.

i) For every feasible point $u$ of \reff{g:kkt:opt},
the tms $[u]_{2k}$ (see \S\ref{sc:pre} for the notation) is feasible for
\reff{momrlx:k}, for all $k$.
Therefore, if \reff{momrlx:k} is infeasible for some $k$,
then \reff{g:kkt:opt} must be infeasible.

ii) By Assumption~\ref{ass:sig}, when \reff{g:kkt:opt} is infeasible, the set
\[
\{x\in \re^n:\, c_{eq}(x)=0, \phi(x)=0, \, \rho(x) \geq 0 \}
\]
is empty. It has a single inequality.
By the Positivstellensatz~\cite[Corollary~4.4.3]{BCR},
it holds that
$
-1 \in  \mbox{Ideal}(c_{eq}, \phi) + \mbox{Qmod}(\rho).
$
By Assumption~\ref{ass:sig},
\[
\mbox{Ideal}(c_{eq}, \phi) + \mbox{Qmod}(\rho) \subseteq
\mbox{IQ}(c_{eq}, c_{in}) + \mbox{IQ}(\phi, \psi).
\]
Thus, for all $k$ big enough, \reff{rlxsos:k} is unbounded from above.
Hence, \reff{momrlx:k} must be infeasible, by weak duality.

When \reff{g:kkt:opt} is feasible,
$f$ achieves finitely many values on $\mc{K}_c$,
so \reff{g:kkt:opt} must achieve its optimal value $f_c$.
By Theorem~\ref{thm:pf:tight},
we know that $f_k = f_k^\prm = f_c$
for all $k$ big enough.
For each minimizer $u^*$ of \reff{g:kkt:opt},
the tms $[u^*]_{2k}$ is a minimizer of \reff{momrlx:k}.

iii) If \reff{ft:Mty} holds, we can get $r := \rank \, M_t (y^*)$
minimizers for \reff{g:kkt:opt} \cite{CuFi05,HenLas05},
say, $u_1, \ldots, u_r$,
such that $f_k = f(u_i)$ for each $i$. Clearly, $f_k = f(u_i) \geq f_c$.
On the other hand, we always have $f_k \leq f_c$.
So, $f_k = f_c$.

iv) By Assumption~\ref{ass:sig}, \reff{g:kkt:opt} is equivalent to the problem
\be \label{opt:rho}
\left\{\baray{rl}
 \min & f(x)  \\
 s.t. &  c_{eq}(x) = 0, \,  \phi(x) = 0, \, \rho(x) \geq 0.
\earay \right.
\ee
The optimal value of \reff{opt:rho} is also $f_c$.
Its $k$th order Lasserre's relaxation is
\be \label{mom:k:rho}
\left\{\baray{rl}
\gamma_k^\prm  := \min & \langle f, y \rangle \\
 s.t. & \langle 1, y \rangle = 1,  M_k(y) \succeq 0, \\
      &  L_{c_{eq}}^{(k)}(y) = 0, \,
       L_{\phi}^{(k)}(y) = 0,   \,  L_{\rho}^{(k)}(y) \succeq 0.
\earay \right.
\ee
Its dual optimization problem is
\be  \label{sos:k:rho}
\left\{\baray{rl}
\gamma_k := \, \max & \gamma  \\
 s.t. &  f-\gamma \in \mbox{Ideal}(c_{eq},\phi)_{2k} +
             \mbox{Qmod}(\rho)_{2k}.
\earay \right.
\ee
By repeating the same proof as for Theorem~\ref{thm:pf:tight}(iii),
we can show that
\[
\gamma_k = \gamma_k^\prm = f_c
\]
for all $k$ big enough. Because
$
\rho \in  \mbox{Qmod}(c_{in}, \psi),
$
each $y$ feasible for \reff{momrlx:k}
is also feasible for \reff{mom:k:rho}.
So, when $k$ is big, each $y^*$ is also a minimizer of \reff{mom:k:rho}.
The problem \reff{opt:rho} also has finitely many minimizers.
By Theorem~2.6 of \cite{Nie-ft},
the condition \reff{ft:Mty} must be satisfied
for some $t\in [d,k]$, when $k$ is big enough.
\end{proof}

If \reff{g:kkt:opt} has infinitely many minimizers, then
the condition \reff{ft:Mty} is typically not satisfied.
We refer to \cite[\S6.6]{Lau}.

\section{Polyhedral constraints}
\label{sc:PH}

In this section, we assume the feasible set of \reff{pop:c(x)>=0} is the polyhedron
\[
P :=\{ x \in \re^n \, \mid \, Ax-b \geq 0\},
\]
where $A = \bbm a_1 & \cdots & a_m \ebm^T \in \re^{m \times n}$,
$b = \bbm b_1 & \cdots & b_m \ebm^T \in \re^m$.
This corresponds to that $\mc{E}=\emptyset$, $\mc{I}=[m]$,
and each $c_i(x) = a_i^Tx - b_i$. Denote
\be \label{PH:DClmd}
D(x) \, := \, \diag( c_1(x), \ldots, c_m(x)), \quad
C(x) \, := \, \bbm  A^T  \\  D(x) \ebm.
\ee
The Lagrange multiplier vector
$\lmd := \bbm \lmd_1 & \cdots & \lmd_m \ebm^T$
satisfies
\be \label{AD*lmd=gf}
\bbm  A^T  \\  D(x) \ebm \lmd = \bbm \nabla f(x) \\ 0 \ebm.
\ee
If $\rank\, A= m$, we can express $\lmd$ as
\be \label{lmd=inv(AAt)Agf}
\lmd =  (AA^T)^{-1} A \nabla f(x).
\ee
If $\rank\, A < m$, how can we express $\lmd$ in terms of $x$?
In computation, we often prefer a polynomial expression.
If there exists $L(x) \in \re[x]^{m \times (n+m)}$ such that
\be \label{polyH:L(x)A(x)=Id}
L(x) C(x)  = I_m,
\ee
then we can get
\[
\lmd \, = \,  L(x) \bbm \nabla f(x) \\ 0 \ebm   \,=\,   L_1(x) \nabla f(x),
\]
where $L_1(x)$ consists of the first $n$ columns of $L(x)$.
In this section, we characterize when such $L(x)$ exists
and give a degree bound for it.

The linear function $Ax-b$ is said to be {\it nonsingular} if
$
\rank \,  C(u) = m
$
for all $u \in \cpx^n$ (also see Definition~\ref{def:nonsig}).
This is equivalent to that for every $u$,
if $J(u) = \{ i_1, \ldots, i_k \}$
(see \reff{df:J(u)} for the notation), then
$
a_{i_1}, \ldots, a_{i_k}
$
are linearly independent.

\begin{pro}  \label{pr:H:Lx}
The linear function $Ax-b$ is nonsingular if and only if
there exists a matrix polynomial $L(x)$ satisfying
\reff{polyH:L(x)A(x)=Id}. Moreover, when $Ax-b$ is nonsingular,
we can choose $L(x)$ in \reff{polyH:L(x)A(x)=Id}
with $\deg(L) \leq m-\rank \,A$.
\end{pro}
\begin{proof}
Clearly, if \reff{polyH:L(x)A(x)=Id} is satisfied by some $L(x)$,
then $\rank\,C(u) \geq m$ for all $u$.
This implies that $Ax-b$ is nonsingular.

Next, assume that $Ax-b$ is nonsingular.
We show that \reff{polyH:L(x)A(x)=Id}
is satisfied by some $L(x) \in \re[x]^{m\times (n+m)}$ with degree
$\leq m-\rank \,A$. Let $r = \rank\,A$.
Up to a linear coordinate transformation,
we can reduce $x$ to a $r$-dimensional variable.
Without loss of generality, we can assume that $\rank\, A = n$
and $m \geq n$.

For a subset $I := \{i_1, \ldots, i_{m-n}\}$ of $[m]$, denote
\[
c_I(x) :=  \prod_{ i \in I } c_i(x), \quad
E_I(x):= c_I(x) \cdot
\diag ( c_{i_1}(x)^{-1}, \ldots,  c_{i_{m-n}}(x)^{-1} ),
\]
\[
D_I(x):=   \diag( c_{i_1}(x), \ldots,  c_{i_{m-n}}(x) ), \quad
A_I = \bbm  a_{i_1} & \cdots & a_{i_{n-m}} \ebm^T.
\]
For the case that $I=\emptyset$ (the empty set),
we set $c_{\emptyset}(x) = 1$. Let
\[
V = \{ I \subseteq [m]: \,\, |I|=m-n, \, \rank \, A_{[m]\backslash I} = n \}.
\]
\noindent
{\bf Step I:} \,
For each $I \in V$, we construct a matrix polynomial
$L_I(x)$ such that
\be \label{eq:LIC=Im}
L_I(x)  C(x)  = c_I(x)  I_m.
\ee
The matrix $L_I:= L_I(x)$ satisfying \reff{eq:LIC=Im}
can be given by the following $2\times 3$ block matrix
($L_I(\mc{J},\mc{K})$ denotes the submatrix
whose row indices are from $\mc{J}$
and whose column indices are from $\mc{K}$):
%
%
%
\be \label{table:LI(x)}
\left(\begin{array}{c|ccc}
\mc{J} \, \backslash \, \mc{K}   &   [n]  &   n+I  &    n + [m] \backslash I \\  \hline
 I   &  0   &   E_I(x)     &      0        \\
\text{$ [m] \backslash I $ }   &  c_I(x) \cdot \big( A_{[m]\backslash I} \big)^{-T}
  & -\big( A_{[m]\backslash I} \big)^{-T} \big( A_{I} \big)^T E_I(x)  & 0     \\
\end{array} \right).
\ee
Equivalently, the blocks of $L_I$ are:
\[
L_I\big(I, [n]\big)= 0, \quad
L_I\big(I, n + [m]\backslash I \big) = 0, \quad
L_I\big([m]\backslash I, n + [m]\backslash I \big) = 0,
\]
\[
L_I \big(I, n+I\big)  = E_I(x), \quad
L_I \big([m]\backslash I, [n]\big)  = c_I(x) \big( A_{[m]\backslash I} \big)^{-T},
\]
\[
L_I \big( [m]\backslash I, n+I \big)  =
-\big( A_{[m]\backslash I} \big)^{-T} \big( A_{I} \big)^T.
\]
For each $I \in V$, $ A_{[m]\backslash I}$ is invertible.
The superscript $^{-T}$ denotes the inverse of the transpose.
Let $G:=L_I(x) C(x)$, then one can verify that
\[
G(I,I) = E_I(x) D_I(x)  = c_I(x)   I_{m-n}, \quad
G(I,[m]\backslash I) = 0,
\]
\[
G([m]\backslash I, [m]\backslash I) =
\bbm   c_I(x) \big( A_{[m]\backslash I} \big)^{-T}
& -A_{[m]\backslash I}^{-T} A_{I}^{T}  E_I(x)   \ebm
\bbm  \big( A_{[m]\backslash I} \big)^T  \\ 0 \ebm
= c_I(x) I_n.
\]
\[
G([m]\backslash I,   I) =
\bbm   c_I(x) \big( A_{[m]\backslash I} \big)^{-T}
& -\big( A_{[m]\backslash I} \big)^{-T} \big( A_{I} \big)^T E_I(x)   \ebm
\bbm A_{I}^T  \\ D_I(x) \ebm  = 0.
\]
This shows that the above $L_I(x)$ satisfies \reff{eq:LIC=Im}.

\bigskip
\noindent
{\bf Step II:} \, We show that there exist real scalars $\nu_I$ satisfying
\be \label{eq:nuIcI=1}
\sum_{ I \in V} \nu_I c_I(x) = 1.
\ee
This can be shown by induction on $m$.
\bit

\item When $m=n$, $V = \emptyset$ and $c_{\emptyset}(x) = 1$,
so \reff{eq:nuIcI=1} is clearly true.

\item When $m > n$,  let
\be \label{set:N}
N  \, := \, \{ i \in [m] \, \mid \,  \rank\, A_{[m]\backslash \{i\} } = n \}.
\ee
For each $i \in N$, let $V_i$ be the set of all
$I' \subseteq [m]\backslash \{i\}$ such that
$|I'|=m-n-1$ and $\rank \, A_{[m]\backslash (I'\cup \{i\}) } = n$.
For each $i \in N$, by the assumption, the linear function
$
A_{m \backslash \{i\} } x - b_{m \backslash \{i\}}
$
is nonsingular. By induction,
there exist real scalars $\nu_{I'}^{(i)}$ satisfying
\be \label{nuIcI:i}
\sum_{ I' \in V_i } \nu_{I'}^{(i)} c_{I'}(x) = 1.
\ee
Since $\rank \, A = n$, we can generally assume that
$\{a_1, \ldots, a_n\}$ is linearly independent.
So, there exist scalars $\af_1, \ldots, \af_n$ such that
\[
a_m = \af_1 a_1 + \cdots + \af_n a_n.
\]
If all $\af_i =0$, then $a_m = 0$,
and hence $A$ can be replaced by its first $m-1$ rows.
So, \reff{eq:nuIcI=1} is true by the induction.
In the following, suppose at least one $\af_i \ne 0$ and write
\[
\{ i: \, \af_i \ne 0 \} =
\{ i_1, \ldots, i_k \}.
\]
%
%
Then, $a_{i_1}, \ldots, a_{i_k}, a_{m}$
are linearly dependent. For convenience, set $i_{k+1} := m$.
Since $Ax-b$ is nonsingular, the linear system
\[
c_{i_1}(x) = \cdots = c_{i_k}(x)= c_{i_{k+1}}(x) = 0
\]
has no solutions. Hence, there exist real scalars
$\mu_1, \ldots, \mu_{k+1}$ such that
\[
\mu_{1} c_{i_1}(x) +  \cdots + \mu_{k} c_{i_k}(x) + \mu_{k+1} c_{i_{k+1}}(x) = 1.
\]
This above can be implied by echelon's form for inconsistent linear systems.
Note that $i_1, \ldots, i_{k+1} \in N$.
For each $j = 1,\ldots, k+1$, by \reff{nuIcI:i},
\[
\sum_{ I' \in V_{i_j} }   \nu_{I'}^{(i_j)} c_{I'}(x) = 1.
\]
Then, we can get
\[
1 =  \sum_{j=1}^{k+1} \mu_{j} c_{i_j}(x) =
\sum_{j=1}^{k+1} \mu_{j} \sum_{ I' \in V_{i_j}}
\nu_{I'}^{(i_j)} c_{i_j}(x) c_{I'}(x)  =
\]
\[
\sum_{ I = I' \cup \{i_j\},  I' \in V_{i_j}, 1 \leq j \leq k+1 }
\nu_{I'}^{(i_j)} \mu_{j}   c_{I}(x).
\]
Since each $I' \cup \{i_j\} \in V$,
\reff{eq:nuIcI=1} must be satisfied by some scalars $\nu_I$.

\eit

\noindent
{\bf Step III:} \, For $L_I(x)$ as in \reff{eq:LIC=Im},
we construct $L(x)$ as
\be \label{L=sum:nuIcILi}
L(x) := \sum_{ I \in V } \nu_I c_I(x) L_I(x).
\ee
Clearly, $L(x)$ satisfies \reff{polyH:L(x)A(x)=Id} because
\[
L(x) C(x) =  \sum_{ I \in V } \nu_I L_I(x) C(x) =
\sum_{ I \in V } \nu_I c_I(x) I_m = I_m.
\]
Each $L_I(x)$ has degree $\le m-n$,
so $L(x)$ has degree $\leq m-n$.
\end{proof}

Proposition~\ref{pr:H:Lx} characterizes
when there exists $L(x)$ satisfying \reff{polyH:L(x)A(x)=Id}.
When it does, a degree bound for $L(x)$ is $m-\rank\,A$.
Sometimes, its degree can be smaller than that,
as shown in Example~\ref{exm:box}.
For given $A,b$, the matrix polynomial $L(x)$
satisfying \reff{polyH:L(x)A(x)=Id}
can be determined by linear equations,
which are obtained by matching coefficients on both sides.
In the following, we give some examples of $L(x)C(x)=I_m$
for polyhedral sets.

\begin{exm}  \label{exm:simplex}
Consider the simplicial set
\[
x_1 \geq 0, \, \ldots, \, x_n \geq 0, \, 1-e^Tx \geq 0.
\]
%
%
The equation $L(x)C(x) = I_{n+1}$ is satisfied by
\[
L(x) =
\bbm
 1 - x_1 &    -x_2 & \cdots  & -x_n   & 1     & \cdots & 1 \\
    -x_1 & 1 - x_2 & \cdots  & -x_n   & 1     & \cdots & 1  \\
  \vdots & \vdots  & \ddots  & \vdots &\vdots & \vdots & \vdots \\
    -x_1 &   -x_2  & \cdots  & 1 -x_n & 1     & \cdots & 1 \\
    -x_1 &   -x_2  & \cdots  &  -x_n  & 1     & \cdots & 1
\ebm.
\]

\end{exm}

\begin{exm}  \label{exm:box}
Consider the box constraint
\[
x_1 \geq 0, \, \ldots, \, x_n \geq 0, \, 1-x_1 \geq 0, \, \ldots, \, 1-x_n \geq 0.
\]
%
%
The equation $L(x)C(x) = I_{2n}$ is satisfied by
\[
L(x) =
\bbm
I_n- \diag(x)  & I_n  &  I_n \\
   - \diag(x)  & I_n  &  I_n \\
\ebm.
\]
\end{exm}

\begin{exm}
Consider the polyhedral set
\[
1-x_4 \geq 0, \, x_4-x_3 \geq 0, \, x_3-x_2 \geq 0, \, x_2-x_1 \geq 0, \, x_1+1 \geq 0.
\]
%
%
The equation $L(x)C(x) = I_{5}$ is satisfied by
\[
L(x) =  \frac{1}{2}
\left[\begin{array}{rrrrrrrrr}  - x_{1} - 1 &  - x_{2} - 1 &  - x_{3} - 1 &  - x_{4} - 1 & 1 & 1 & 1 & 1 & 1\\  - x_{1} - 1 &  - x_{2} - 1 &  - x_{3} - 1 & 1 - x_{4} & 1 & 1 & 1 & 1 & 1\\  - x_{1} - 1 &  - x_{2} - 1 & 1 - x_{3} & 1 - x_{4} & 1 & 1 & 1 & 1 & 1\\  - x_{1} - 1 & 1 - x_{2} & 1 - x_{3} & 1 - x_{4} & 1 & 1 & 1 & 1 & 1\\ 1 - x_{1} & 1 - x_{2} & 1 - x_{3} & 1 - x_{4} & 1 & 1 & 1 & 1 & 1 \end{array}\right].
\]

\end{exm}

\begin{exm}
Consider the polyhedral set
\[
1+x_1 \geq 0, \, 1-x_1 \geq 0, \, 2-x_1-x_2 \geq 0, \, 2-x_1 +x_2 \geq 0.
\]
%
%
The matrix $L(x)$ satisfying $L(x)C(x) = I_4$ is {\small
\[
\frac {1}{6}
\left[\begin{array}{rrrrrr} {x_{1}}^2 - 3\, x_{1} + 2 & x_{1}\, x_{2} - x_{2} & 4 - x_{1} & 2 - x_{1} & 1 - x_{1} & 1 - x_{1}\\ 3\, {x_{1}}^2 - 3\, x_{1} - 6 & 3\, x_{2} + 3\, x_{1}\, x_{2} & 6 - 3\, x_{1} & - 3\, x_{1} &  - 3\, x_{1} - 3 &  - 3\, x_{1} - 3\\ 1 - {x_{1}}^2 &  - 2\, x_{2} - x_{1}\, x_{2} - 3 & x_{1} - 1 & x_{1} + 1 & x_{1} + 2 & x_{1} + 2\\ 1 - {x_{1}}^2 & 3 - x_{1}\, x_{2} - 2\, x_{2} & x_{1} - 1 & x_{1} + 1 & x_{1} + 2 & x_{1} + 2 \end{array}\right].
\]
}

\end{exm}

\section{General constraints}
\label{sc:gencon}

We consider general nonlinear constraints as in \reff{pop:c(x)>=0}.
The critical point conditions are in \reff{eq:cfx:lmd}.
%
%
We discuss how to express Lagrange multipliers $\lmd_i$
as polynomial functions in $x$
on the set of critical points.

Suppose there are totally $m$
equality and inequality constraints, i.e.,
\[
\mc{E} \cup \mc{I} = \{1, \ldots, m \}.
\]
If $(x,\lmd)$ is a critical pair, then
$\lmd_i c_i(x) = 0$ for all $i \in \mc{E} \cup \mc{I}$.
So, the Lagrange multiplier vector
$\lmd := \bbm \lmd_1 & \cdots & \lmd_m \ebm^T$
satisfies the equation
\be \label{C(x)*lmd=gf(x)}
\underbrace{
\bbm
\nabla c_1(x) & \nabla c_2(x) & \cdots & \nabla c_m(x) \\
 c_1(x) & 0 & \cdots & 0 \\
 0   &  c_2(x) & 0  & 0 \\
 \vdots & \vdots & \ddots & \vdots \\
 0  &  0  & \cdots &  c_m(x)  \\
\ebm
}_{C(x)}
\lmd  =
\bbm \nabla f(x) \\ 0 \\ 0 \\ \vdots \\ 0  \ebm.
\ee
Let $C(x)$ be as in above. If there exists
$L(x) \in \re[x]^{m \times (m+n)}$ such that
\be \label{eq:L(x)C(x)=Im}
L(x) C(x) = I_m,
\ee
then we can get
\be \label{lmd=L1*gf(x)}
\lmd   = L(x)  \bbm \nabla f(x) \\ 0 \ebm =    L_1(x)  \nabla f(x),
\ee
where $L_1(x)$ consists of the first $n$ columns of $L(x)$.
Clearly, \reff{eq:L(x)C(x)=Im}
implies that Assumption~\ref{ass:lmd:p} holds.
This section characterizes when such $L(x)$ exists.

\begin{defi} \label{def:nonsig}
The tuple $c := (c_1, \ldots, c_m)$ of constraining polynomials
is said to be {\it nonsingular} if
$\rank\, C(u) = m$ for every $u \in \cpx^n$.
\end{defi}

Clearly, $c$ being nonsingular
is equivalent to that for each $u \in \cpx^n$,
if $J(u) = \{ i_1, \ldots, i_k \}$
(see \reff{df:J(u)} for the notation), then the gradients
$
\nabla c_{i_1}(u), \ldots, \nabla c_{i_k}(u)
$
are linearly independent.
Our main conclusion is that \reff{eq:L(x)C(x)=Im}
holds if and only if the tuple $c$ is nonsingular.

\begin{prop}  \label{pro:L(x)W(x)=Idt}
(i) For each $W(x) \in \cpx[x]^{s \times t}$ with $s \geq t$,
$\rank\,W(u) = t$ for all $u \in \cpx^n$
if and only if there exists
$P(x) \in \cpx[x]^{t \times s}$ such that
\[
P(x) W(x) = I_t.
\]
Moreover, for $W(x) \in \re[x]^{s \times t}$, we can choose
$P(x) \in \re[x]^{t \times s}$ for the above. \\
(ii) The constraining polynomial tuple
$c$ is nonsingular if and only if
there exists $L(x) \in \re[x]^{m \times (m+n)}$
satisfying \reff{eq:L(x)C(x)=Im}.
\end{prop}
\begin{proof}
(i) ``$\Leftarrow$":\,
If $L(x) W(x) = I_t$, then for all $u\in \cpx^n$
\[
t = \rank \, I_t \leq \rank \, W(u) \leq t.
\]
So, $W(x)$ must have full column rank everywhere.

\bigskip
\noindent
``$\Rightarrow$":\,
Suppose $\rank\,W(u) = t$ for all $u\in\cpx^n$. Write $W(x)$ in columns
\[
W(x) = \bbm  w_1(x)  &  w_2(x) & \cdots & w_t(x) \ebm.
\]
Then, the equation $w_1(x)=0$
does not have a complex solution.
By Hilbert's Weak Nullstellensatz \cite{CLO97},
there exists $\xi_1(x) \in \cpx[x]^s$ such that
$\xi_1(x)^T w_1(x) = 1.$
For each $i = 2, \ldots, t$, denote
\[
r_{1,i}(x):= \xi_1(x)^T w_i(x),
\]
then (use $\sim$ to denote row equivalence between matrices)
\[
W(x) \sim
\bbm
1 & r_{1,2}(x) & \cdots & r_{1,t}(x) \\
w_1(x) & w_2(x) & \cdots & w_m(x) \\
\ebm
\sim
W_1(x) :=
\bbm
 1 & r_{1,2}(x) & \cdots & r_{1,m}(x) \\
0  & w_2^{(1)}(x) & \cdots & w_m^{(1)}(x) \\
\ebm,
\]
where each ($i=2,\ldots,m$)
\[
w_i^{(1)}(x) =  w_i(x) - r_{1,i}(x) w_1(x).
\]
So, there exists $P_1(x) \in \re[x]^{(s+1) \times s}$ such that
\[
P_1(x) W(x) = W_1(x).
\]
Since $W(x)$ and $W_1(x)$ are row equivalent,
$W_1(x)$ must also have full column rank everywhere.
Similarly, the polynomial equation
\[
w_2^{(1)}(x) = 0
\]
does not have a complex solution. Again,
by Hilbert's Weak Nullstellensatz \cite{CLO97},
there exists $\xi_2(x) \in \cpx[x]^{s}$ such that
\[
\xi_2(x)^T w_2^{(1)}(x) = 1.
\]
For each $i = 3, \ldots, t$, let $r_{2,i}(x):= \xi_2(x)^T w_2^{(1)}(x)$, then
\[
W_1(x) \sim
\bbm
1 & r_{1,2}(x) &  r_{1,3}(x) & \cdots & r_{1,m}(x) \\
0 & 1  & r_{2,3}(x) & \cdots & r_{2,m}(x) \\
0  & w_2^{(1)}(x) & w_3^{(1)}(x) & \cdots & w_m^{(1)}(x) \\
\ebm
\sim
\]
\[
W_2(x) :=
\bbm
1 & r_{1,2}(x) &  r_{1,3}(x) & \cdots & r_{1,m}(x) \\
0 & 1  & r_{2,3}(x) & \cdots & r_{2,m}(x) \\
0  & 0 & w_3^{(2)}(x) & \cdots & w_m^{(2)}(x) \\
\ebm,
\]
where each ($i=3,\ldots,m$)
\[
w_i^{(2)}(x) =  w_i^{(1)}(x) - r_{2,i}(x) w_2^{(1)}(x).
\]
Similarly, $W_1(x)$ and $W_2(x)$ are row equivalent,
so $W_2(x)$ has full column rank everywhere.
There exists $P_2(x) \in \cpx[x]^{(s+2) \times (s+1)}$ such that
\[
P_2(x) W_1(x) = W_2(x).
\]
Continuing this process, we can finally get
\[
W_2(x) \sim \cdots \sim
W_t(x):=
\bbm
1 & r_{1,2}(x) &  r_{1,3}(x) & \cdots & r_{1,t}(x) \\
0 & 1  & r_{2,3}(x) & \cdots & r_{2,t}(x) \\
0  & 0 & 1 & \cdots & r_{3,t}(x) \\
\vdots &  \vdots & \vdots & \ddots & \vdots \\
0  &  0  &  0  & \cdots &  1 \\
0  &  0  &  0  & \cdots &  0 \\
\ebm.
\]
Consequently, there exists $P_i(x) \in \re[x]^{(s+i) \times (s+i-1)}$
for $i=1,2,\ldots, t$, such that
\[
P_t(x) P_{t-1}(x) \cdots P_1(x) W(x) = W_t(x).
\]
Since $W_t(x)$ is a unit upper triangular matrix polynomial,
there exists $P_{t+1}(x) \in \re[x]^{t \times (s+t)}$ such that
$
P_{t+1}(x) W_t(x) = I_t.
$
Let
\[
P(x) \, := \, P_{t+1}(x) P_t(x) P_{t-1}(x) \cdots P_1(x),
\]
then $P(x) W(x) = I_m$.
Note that $P(x) \in \cpx[x]^{t \times  s}$.

\smallskip
For $W(x) \in \re[x]^{s\times t}$, we can replace $P(x)$ by
$\big(P(x) + \overline{P(x)}\big)/2$
(the $\overline{P(x)}$ denotes the complex conjugate of $P(x)$),
which is a real matrix polynomial.

\smallskip
(ii) The conclusion is implied directly by the item (i).
\end{proof}

In Proposition~\ref{pro:L(x)W(x)=Idt},
there is no explicit degree bound for $L(x)$
satisfying \reff{eq:L(x)C(x)=Im}.
This question is mostly open, to the best of the author's  knowledge.
However, once a degree is chosen for $L(x)$, it can be determined
by comparing coefficients of both sides of \reff{eq:L(x)C(x)=Im}.
This can be done by solving a linear system. In the following, we give some
examples of $L(x)$ satisfying \reff{eq:L(x)C(x)=Im}.

\begin{exm}
Consider the hypercube with quadratic constraints
\[
1-x_1^2 \geq 0, 1-x_2^2 \geq 0, \ldots, 1-x_n^2 \geq 0.
\]
The equation $L(x)C(x)=I_n$ is satisfied by
\[
L(x) = \bbm -\frac{1}{2} \diag(x)  &  I_n \ebm.
\]
\end{exm}

\begin{exm}  \label{exm:nng:sph}
Consider the nonnegative portion of the unit sphere
\[
x_1 \geq 0, x_2 \geq 0, \ldots, x_n\geq 0,  x_1^2+\cdots+x_n^2 - 1 =0.
\]
%
%
The equation $L(x)C(x)=I_{n+1}$ is satisfied by
\[
L(x) =
\bbm
I_n - xx^T &    x \mbf{1}_n^T    &   2x \\
\half x^T  & -\half \mbf{1}_n^T  &   -1
\ebm.
\]
\end{exm}

\begin{exm}
Consider the set
\[
1-x_1^3-x_2^4\geq 0,  1- x_3^4 - x_4^3 \geq 0.
\]
The equation $L(x)C(x)=I_2$ is satisfied by
\[
L(x) =
\left[\begin{array}{cccccc} -\frac{x_{1}}{3} & -\frac{x_{2}}{4} & 0 & 0 & 1 & 0\\ 0 & 0 & -\frac{x_{3}}{4} & -\frac{x_{4}}{3} & 0 & 1 \end{array}\right].
\]
\end{exm}

\begin{exm}
Consider the quadratic set
\[
1-x_1x_2 -x_2x_3 - x_1x_3 \geq 0, \,  1- x_1^2 - x_2^2 - x_3^2 \geq 0.
\]
The matrix $L(x)^T$ satisfying $L(x)C(x)=I_2$ is
{\smaller \smaller \smaller
\[
\left[\begin{array}{rr} 25\, {x_{1}}^3 + 10\, {x_{1}}^2\, x_{2} + 40\, x_{1}\, {x_{2}}^2 - 25\, x_{1} - 2\, x_{3} &  - 25\, {x_{1}}^3 - 10\, {x_{1}}^2\, x_{2} - 40\, x_{1}\, {x_{2}}^2 + \frac{49\, x_{1}}{2} + 2\, x_{3}\\  - 15\, {x_{1}}^2\, x_{2} + 10\, x_{1}\, {x_{2}}^2 + 20 \,x_{3}\, x_{1}\, x_{2} - 10\, x_{1} & 15\, {x_{1}}^2\, x_{2} - 10\, x_{1}\, {x_{2}}^2 - 20\, x_{3}\, x_{1}\, x_{2} + 10\, x_{1} - \frac{x_{2}}{2}\\ 25\, x_{3}\, {x_{1}}^2 - 20\, x_{1}\, {x_{2}}^2 + 10\, x_{3}\, x_{1}\, x_{2} + 2\, x_{1} &  - 25\, x_{3}\, {x_{1}}^2 + 20\, x_{1}\, {x_{2}}^2 - 10\, x_{3}\, x_{1}\, x_{2} - 2\, x_{1} - \frac{x_{3}}{2}\\ 1 - 20\, x_{1}\, x_{3} - 10\, {x_{1}}^2 - 20\, x_{1}\, x_{2} & 20\, x_{1}\, x_{2} + 20\, x_{1}\, x_{3} + 10\, {x_{1}}^2\\  - 50\, {x_{1}}^2 - 20\, x_{2}\, x_{1} & 50\, {x_{1}}^2 + 20\, x_{2}\, x_{1} + 1 \end{array}\right].
\]
}
\end{exm}

We would like to remark that a polynomial tuple $c=(c_1, \ldots, c_m)$
is generically nonsingular and Assumption~\ref{ass:lmd:p}
holds generically.

\begin{prop} \label{pro:c-ns-gen}
For all positive degrees $d_1, \ldots, d_m$,
there exists an open dense subset $\mc{U}$ of
$\mc{D} := \re[x]_{d_1} \times \cdots \times \re[x]_{d_m}$
such that every tuple $c=(c_1, \ldots, c_m) \in \mc{U}$
is nonsingular. Indeed, such $\mc{U}$ can be chosen
as a Zariski open subset of $\mc{D}$, i.e.,
it is the complement of a proper real variety of $\mc{D}$.
Moreover, Assumption~\ref{ass:lmd:p} holds for all $c \in \mc{U}$,
i.e., it holds generically.
\end{prop}
\begin{proof}
The proof needs to use resultants and discriminants,
which we refer to \cite{Nie-dis}.

First, let $J_1$ be the set of all
$(i_1, \ldots, i_{n+1})$ with
$1 \leq i_1 < \cdots < i_{n+1} \leq m$.
The resultant $Res(c_{i_1}, \ldots, c_{i_{n+1}})$
\cite[Section~2]{Nie-dis} is a polynomial
in the coefficients of $c_{i_1}, \ldots, c_{i_{n+1}}$
such that if $Res(c_{i_1}, \ldots, c_{i_{n+1}}) \ne 0$
then the equations
\[
c_{i_1}(x) =\cdots = c_{i_{n+1}}(x) = 0
\]
have no complex solutions. Define
\[
F_1 (c) \, := \, \prod_{ (i_1, \ldots, i_{n+1}) \in J_1  }
Res(c_{i_1}, \ldots, c_{i_{n+1}}).
\]
For the case that $m\leq n$, $J_1 = \emptyset$
and we just simply let $F_1(c) = 1$.
Clearly, if $F_1(c) \ne 0$, then no more than $n$ polynomials
of $c_1, \ldots, c_m$ have a common complex zero.

Second, let $J_2$ be the set of all
$(j_1, \ldots, j_k)$ with $k \leq n$ and
$1 \leq j_1 < \cdots < j_k \leq m$.
When one of $c_{j_1}, \ldots, c_{j_k}$ has degree bigger than one,
the discriminant $\Delta(c_{j_1}, \ldots, c_{j_k})$ is a polynomial
in the coefficients of $c_{j_1}, \ldots, c_{j_k}$
such that if $\Delta(c_{j_1}, \ldots, c_{j_k}) \ne 0$
then the equations
\[
c_{j_1}(x) =\cdots = c_{j_k}(x) = 0
\]
have no singular complex solution \cite[Section~3]{Nie-dis}, i.e.,
at every complex common solution $u$,
the gradients of $c_{j_1}, \ldots, c_{j_k}$
at $u$ are linearly independent.
When all $c_{j_1}, \ldots, c_{j_k}$ have degree one,
the discriminant of the tuple $(c_{j_1}, \ldots, c_{j_k})$
is not a single polynomial,
but we can define $\Delta(c_{j_1}, \ldots, c_{j_k})$
to be the product of all maximum minors of
its Jacobian (a constant matrix). Define
\[
F_2 (c) \, := \, \prod_{ (j_1, \ldots, j_k) \in J_2  }
\Delta(c_{j_1}, \ldots, c_{j_k}).
\]
Clearly, if $F_2(c) \ne 0$, then
then no $n$ or less polynomials
of $c_1, \ldots, c_m$ have a singular complex comon zero.

Last, let $F(c) := F_1(c) F_2(c)$ and
\[
\mc{U} \, := \, \{ c=(c_1, \ldots, c_m) \in \mc{D}: \,
F(c) \ne 0 \}.
\]
Note that $\mc{U}$ is a Zariski open subset of $\mc{D}$
and it is open dense in $\mc{D}$. For all $c\in \mc{D}$,
no more than $n$ of $c_1, \ldots, c_m$ can have a complex common zero.
For any $k$ polynomials ($k \leq n$) of $c_1, \ldots, c_m$,
if they have a complex common zero, say, $u$,
then their gradients at $u$ must be linearly independent.
This means that $c$ is a nonsingular tuple.

Since every $c \in \mc{U}$ is nonsingular, Proposition~\ref{pro:L(x)W(x)=Idt}
implies \reff{eq:L(x)C(x)=Im}, whence Assumption~\ref{ass:lmd:p} is satisifed.
Therefore, Assumption~\ref{ass:lmd:p} holds
for all $c \in \mc{U}$. So, it holds generically.
\end{proof}

\section{Numerical examples}
\label{sc:exm}

This section gives examples of using the new relaxations
\reff{momrlx:k}-\reff{rlxsos:k}
for solving the optimization problem~\reff{pop:c(x)>=0},
with usage of Lagrange multiplier expressions.
Some polynomials in the examples are from \cite{Rez00}.
The computation is implemented in MATLAB R2012a,
on a Lenovo Laptop with CPU@2.90GHz and RAM 16.0G.
The relaxations \reff{momrlx:k}-\reff{rlxsos:k} are solved by the software
{\tt GloptiPoly~3} \cite{GloPol3},
which calls the SDP package {\tt SeDuMi} \cite{Sturm}.
For neatness, only four decimal digits are displayed
for computational results.

The polynomials $p_i$ in Assumption~\ref{ass:lmd:p}
are constructed as follows. Order the constraining polynomials
as $c_1, \ldots, c_m$. First, find a matrix polynomial
$L(x)$ satisfying \reff{polyH:L(x)A(x)=Id} or \reff{eq:L(x)C(x)=Im}.
Let $L_1(x)$ be the submatrix of $L(x)$, consisting of the first $n$ columns.
Then, choose $(p_1, \ldots, p_m)$ to be the product
$L_1(x) \nabla f(x)$, i.e.,
\[
p_i = \Big( L_1(x) \nabla f(x) \Big)_i.
\]
In all our examples, the global minimum value
$f_{\min}$ of \reff{pop:c(x)>=0} is achieved at a critical point.
This is the case if the feasible set is compact,
or if $f$ is coercive (i.e.,the sublevel set
$\{f(x) \leq \ell \}$ is compact for all $\ell$),
and the constraint qualification condition holds.

By Theorem~\ref{thm:pf:tight}, we have
$f_k = f_{\min}$ for all $k$ big enough,
if $f_c = f_{\min}$ and
anyone of its conditions i)-iii) holds.
Typically, it might be inconvenient to check these conditions.
However, in computation, we do not need to check them at all.
Indeed, the condition~\reff{ft:Mty} is more convenient for usage.
When there are finitely many global minimizers,
Theorem~\ref{thm:ft} proved that \reff{ft:Mty}
is an appropriate criteria for detecting convergence.
It is satisfied for all our examples,
except Examples~\ref{exmp:01}, \ref{exmp:07} and \ref{exmp:09}
(they have infinitely many minimizers).

We compare the new relaxations~\reff{momrlx:k}-\reff{rlxsos:k}
with standard Lasserre relaxations in \cite{Las01}.
The lower bounds given by relaxations in \cite{Las01}
(without using Lagrange multiplier expressions)
and the lower bounds given by \reff{momrlx:k}-\reff{rlxsos:k}
(using Lagrange multiplier expressions)
are shown in the tables.
The computational time (in seconds) is also compared.
The results for standard Lasserre relaxations
are titled ``w./o. L.M.E.",
and those for the new relaxations~\reff{momrlx:k}-\reff{rlxsos:k}
are titled ``with L.M.E.".
%
%

\begin{exm} \label{exmp:01}
Consider the optimization problem
\[
\left\{\baray{rl}
\min & x_1x_2(10 - x_3 )\\
s.t. &  x_1\geq 0, x_2 \geq 0, x_3 \geq 0, 1-x_1-x_2-x_3 \geq 0.
\earay\right.
\]
The matrix polynomial $L(x)$ is given in Example~\ref{exm:simplex}.
Since the feasible set is compact, the minimum $f_{\min} =0$
is achieved at a critical point.
The condition~ii) of Theorem~\ref{thm:pf:tight}
is satisfied.\footnote{Note that
$1-x^Tx =(1-e^Tx)(1+x^Tx)+\sum_{i=1}^n x_i(1-x_i)^2+\sum_{i\neq j}x_i^2x_j
\in \mbox{IQ}(c_{eq},c_{in}).$
}
Each feasible point $(x_1,x_2,x_3)$ with $x_1x_2=0$ is a global minimizer.
The computational results for standard Lasserre's relaxations
and the new ones \reff{momrlx:k}-\reff{rlxsos:k} are in Table~\ref{tab:emp:01}.
It confirms that $f_k = f_{\min}$ for all $k\geq 3$,
up to numerical round-off errors.
\begin{table}[htb]
\centering
\caption{Computational results for Example~\ref{exmp:01}.}
\label{tab:emp:01}
\begin{tabular}{|c|c|c|c|c|} \hline
\multirow{2}{*}{order $k$} &  \multicolumn{2}{c|}{w./o. L.M.E.}
         &  \multicolumn{2}{c|}{with L.M.E.}   \\ \cline{2-5}
  &  lower bound &  time  &  lower bound &  time   \\ \hline
2 &  $-0.0521$   &  $0.6841$   & $-0.0521$           &  $0.1922$   \\   \hline
3 &  $-0.0026$   &  $0.2657$   & $-3 \cdot 10^{-8}$  &  $0.2285$    \\   \hline
4 &  $-0.0007$   &  $0.6785$   & $-6 \cdot 10^{-9}$  &  $0.4431$    \\   \hline
5 &  $-0.0004$   &  $1.6105$   & $-2 \cdot 10^{-9}$  &  $0.9567$    \\   \hline
\end{tabular}
\end{table}

\end{exm}

\begin{exm} \label{exmp:02}
Consider the optimization problem
\[
\left\{\baray{rl}
\min & x_1^4x_2^2 + x_1^2x_2^4 + x_3^6 - 3 x_1^2x_2^2 x_3^2 + (x_1^4 + x_2^4 + x_3^4) \\
s.t. &  x_1^2 + x_2^2 +  x_3^2 \geq 1.
\earay\right.
\]
The matrix polynomial $L(x) = \bbm \half x_1 & \half x_2 & \half x_3  & -1 \ebm$.
The objective $f$ is the sum of the positive definite form $x_1^4+x_2^4+x_3^4$
and the Motzkin polynomial
\[
M(x) := x_1^4x_2^2 + x_1^2x_2^4 + x_3^6 - 3 x_1^2x_2^2 x_3^2.
\]
Note that $M(x)$ is nonnegative everywhere but not SOS \cite{Rez00}.
Clearly, $f$ is coercive and $f_{\min}$
is achieved at a critical point.
The set $\mbox{IQ}(\phi,\psi)$ is archimedean, because
\[
c_1(x)p_1(x) = (x_1^2 + x_2^2 +  x_3^2 -1)\big(3M(x) +
2(x_1^4 + x_2^4 + x_3^4)\big) = 0
\]
defines a compact set. So, the condition~i) of Theorem~\ref{thm:pf:tight}
is satisfied.\footnote{
This is because $-c_1^2p_1^2 \in \mbox{Ideal}(\phi) \subseteq \mbox{IQ}(\phi,\psi)$
and the set $\{-c_1(x)^2p_1(x)^2 \geq 0 \}$ is compact.
}
The minimum value $f_{\min}= \frac{1}{3}$,
and there are $8$ minimizers
$
(\pm \frac{1}{\sqrt{3}}, \, \pm \frac{1}{\sqrt{3}}, \, \pm \frac{1}{\sqrt{3}}).
$
The computational results for standard Lasserre's relaxations
and the new ones \reff{momrlx:k}-\reff{rlxsos:k} are in Table~\ref{tab:emp:02}.
It confirms that $f_k = f_{\min}$ for all $k\geq 4$,
up to numerical round-off errors.
\begin{table}[htb]
\centering
\caption{Computational results for Example~\ref{exmp:02}.}
\label{tab:emp:02}
\begin{tabular}{|c|c|c|c|c|} \hline
\multirow{2}{*}{order $k$} &  \multicolumn{2}{c|}{w./o. L.M.E.}
         &  \multicolumn{2}{c|}{with L.M.E.}   \\ \cline{2-5}
  &  lower bound &  time   &  lower bound &  time   \\ \hline
3 & $-\infty$    & $0.4466$  & $0.1111$   & $0.1169$   \\   \hline
4 & $-\infty$    & $0.4948$  & $0.3333$   & $0.3499$   \\   \hline
5 & $ -2.1821 \cdot 10^5$    & $1.1836$  & $0.3333$   & $0.6530$   \\  \hline
\end{tabular}
\end{table}

\end{exm}

\begin{exm}\label{exmp:03}
Consider the optimization problem:
\[
\left\{\baray{rl}
\min &  x_1x_2 + x_2x_3 + x_3x_4 - 3 x_1x_2 x_3x_4 +( x_1^3+\cdots + x_4^3)   \\
s.t. & x_1, x_2, x_3, x_4  \geq  0, 1-x_1-x_2 \geq 0, 1 - x_3 - x_4 \geq 0.
\earay\right.
\]
The matrix polynomial $L(x)$ is
\[
\left[\begin{array}{rrrrrrrrrr} 1 - x_{1} & - x_{2} & 0 & 0 & 1 & 1 & 0 & 0 & 1 & 0\\ - x_{1} & 1 - x_{2} & 0 & 0 & 1 & 1 & 0 & 0 & 1 & 0\\ 0 & 0 & 1 - x_{3} & - x_{4} & 0 & 0 & 1 & 1 & 0 & 1\\ 0 & 0 & - x_{3} & 1 - x_{4} & 0 & 0 & 1 & 1 & 0 & 1\\ - x_{1} & - x_{2} & 0 & 0 & 1 & 1 & 0 & 0 & 1 & 0\\ 0 & 0 & - x_{3} & - x_{4} & 0 & 0 & 1 & 1 & 0 & 1 \end{array}\right].
\]
The feasible set is compact, so
$f_{\min}$ is achieved at a critical point.
One can show that $f_{\min} =0$ and the minimizer is the origin.
The condition~ii) of Theorem~\ref{thm:pf:tight}
is satisfied, because $\mbox{IQ}(c_{eq},c_{in})$
is archimedean.\footnote{
This is because $1-x_1^2-x_2^2$ belongs to the quadratic module
of $(x_1, x_2,1-x_1-x_2)$ and $1-x_3^2-x_4^2$ belongs to the quadratic module
of $(x_3, x_4,1-x_3-x_4)$.
See the footnote in Example~\ref{exmp:01}.
}
The computational results for standard Lasserre's relaxations
and the new ones \reff{momrlx:k}-\reff{rlxsos:k} are in Table~\ref{tab:emp:03}.
\begin{table}[htb]
\centering
\caption{Computational results for Example~\ref{exmp:03}.}
\label{tab:emp:03}
\begin{tabular}{|c|c|r|c|r|} \hline
\multirow{2}{*}{order $k$} &  \multicolumn{2}{c|}{w./o. L.M.E.}
         &  \multicolumn{2}{c|}{with L.M.E.}   \\ \cline{2-5}
  &  lower bound &  time \,\,  &  lower bound &  time \,\,   \\ \hline
3 &  $-2.9 \cdot 10^{-5}$       &  $0.7335$   & $-6 \cdot 10^{-7}$   &  $0.6091$   \\   \hline
4 &  $-1.4 \cdot 10^{-5}$    &  $2.5055$   & $-8 \cdot 10^{-8}$   &  $2.7423$   \\   \hline
5 &   $-1.4 \cdot 10^{-5}$   & $12.7092$   & $-5 \cdot 10^{-8}$   & $13.7449$    \\   \hline
\end{tabular}
\label{tab:exmp:3}
\end{table}
\end{exm}

\begin{exm}\label{exmp:04}
Consider the polynomial optimization problem
\[
\left\{ \baray{rl}
\underset{x\in\re^2}{\min} &  x_1^2+50x_2^2  \\
s.t. & x_1^2- \frac{1}{2} \geq 0,
x_2^2-2x_1x_2 - \frac{1}{8} \geq 0, x_2^2+2x_1x_2 - \frac{1}{8} \geq 0.
\earay \right.
\]
It is motivated from an example in \cite[\S3]{HLNZ}.
The first column of $L(x)$ is
\[
\left[\begin{array}{r} \frac{8\, {x_{1}}^3}{5} + \frac{x_{1}}{5}\\ \frac{288\, x_{2}\, {x_{1}}^4}{5} - \frac{16\, {x_{1}}^3}{5} - \frac{x_{2}\, {x_{1}}^2\, 124}{5} + \frac{8\, x_{1}}{5} - 2\, x_{2}\\  - \frac{288\, x_{2}\, {x_{1}}^4}{5} - \frac{16\, {x_{1}}^3}{5} + \frac{x_{2}\, {x_{1}}^2\, 124}{5} + \frac{8\, x_{1}}{5} + 2\, x_{2} \end{array}\right],
\]
and the second column of $L(x)$ is
\[
\left[\begin{array}{r}  - \frac{8\, {x_{1}}^2\, x_{2}}{5} + \frac{4\, {x_{2}}^3}{5} - \frac{x_{2}}{10}\\ \frac{288\, {x_{1}}^3\, {x_{2}}^2}{5} + \frac{16\, {x_{1}}^2\, x_{2}}{5} - \frac{142\, x_{1}\, {x_{2}}^2}{5} - \frac{9\, x_{1}}{20} - \frac{8\, {x_{2}}^3}{5} + \frac{11\, x_{2}}{5}\\  - \frac{288\, {x_{1}}^3\, {x_{2}}^2}{5} + \frac{16\, {x_{1}}^2\, x_{2}}{5} + \frac{142\, x_{1}\, {x_{2}}^2}{5} + \frac{9\, x_{1}}{20} - \frac{8\, {x_{2}}^3}{5} + \frac{11\, x_{2}}{5} \end{array}\right].
\]
%
%
The objective is coercive, so $f_{\min}$
is achieved at a critical point. The minimum value
$f_{\min} = 56+3/4+25\sqrt{5}  \approx 112.6517$ and
the minimizers are $(\pm \sqrt{1/2}, \pm (\sqrt{5/8}+\sqrt{1/2}))$.
The computational results for standard Lasserre's relaxations
and the new ones \reff{momrlx:k}-\reff{rlxsos:k} are in Table~\ref{tab:emp:04}.
It confirms that $f_k = f_{\min}$ for all $k\geq 4$,
up to numerical round-off errors.
\begin{table}[htb]
\centering
\caption{Computational results for Example~\ref{exmp:04}.}
\label{tab:emp:04}
\begin{tabular}{|c|c|c|c|c|} \hline
\multirow{2}{*}{order $k$} &  \multicolumn{2}{c|}{w./o. L.M.E.}
         &  \multicolumn{2}{c|}{with L.M.E.}   \\ \cline{2-5}
  &  lower bound &  time  &  lower bound &  time   \\ \hline
3 &  $6.7535$     &  $0.4611$    &  $56.7500$     & $0.1309$    \\   \hline
4 &  $6.9294$     &  $0.2428$    &  $112.6517$    & $0.2405$    \\   \hline
5 &  $8.8519$     &  $0.3376$    &  $112.6517$    & $0.2167$    \\   \hline
6 &  $16.5971$    &  $0.4703$    &  $112.6517$    & $0.3788$    \\   \hline
7 &  $35.4756$    &  $0.6536$    &  $112.6517$    & $0.4537$    \\   \hline
\end{tabular}
\end{table}

\end{exm}

\begin{exm}\label{exmp:05}
Consider the optimization problem
\[
\left\{\baray{rl}
\min\limits_{x\in\re^3} &  x_1^3+x_2^3+x_3^3 + 4x_1x_2x_3
-\big( x_1(x_2^2+x_3^2)+x_2(x_3^2+x_1^2)+x_3(x_1^2+x_2^2) \big) \\
\mbox{s.t.} &  x_1 \geq 0, x_1x_2-1 \geq 0, \, x_2x_3 -1 \geq 0.
\earay\right.
\]
The matrix polynomial $L(x)$ is
\[
\left[
\begin{array}{rrrrrr} 1 - x_{1}\, x_{2} & 0 & 0 & x_{2} & x_{2} & 0\\ x_{1} & 0 & 0 & -1 & -1 & 0\\ - x_{1} & x_{2} & 0 & 1 & 0 & -1 \end{array}
\right].
\]
The objective is a variation of Robinson's form \cite{Rez00}.
It is a positive definite form over the nonnegative orthant $\re_+^3$,
so the minimum value is achieved at a critical point.
In computation, we got $f_{\min} \approx 0.9492$
and a global minimizer $( 0.9071, 1.1024, 0.9071)$.
The computational results for standard Lasserre's relaxations
and the new ones \reff{momrlx:k}-\reff{rlxsos:k} are in Table~\ref{tab:emp:05}.
It confirms that $f_k = f_{\min}$ for all $k\geq 3$,
up to numerical round-off errors.
\begin{table}[htb]
\centering
\caption{Computational results for Example~\ref{exmp:05}.}
\label{tab:emp:05}
\begin{tabular}{|c|c|c|c|c|} \hline
\multirow{2}{*}{order $k$} &  \multicolumn{2}{c|}{w./o. L.M.E.}
         &  \multicolumn{2}{c|}{with L.M.E.}   \\ \cline{2-5}
  &  lower bound &  time  &  lower bound &  time   \\ \hline
2 & $-\infty$              &  $0.4129$    &  $-\infty$   & $0.1900$    \\   \hline
3 & $-7.8184 \cdot 10^6$   &  $0.4641$    &  $0.9492$    & $0.3139$    \\   \hline
4 & $-2.0575 \cdot 10^4$   &  $0.6499$    &  $0.9492$    & $0.5057$    \\   \hline
\end{tabular}
\end{table}
\end{exm}

\begin{exm}\label{exmp:06}
Consider the optimization problem ($x_0 := 1$)
\[
\left\{\baray{rl}
\min\limits_{x\in\re^4} &
x^Tx + \sum_{i=0}^4\prod\limits_{j\ne i} (x_i-x_j) \\
\mbox{s.t.} &  x_1^2-1 \geq 0,\, x_2^2 -1 \geq 0,\, x_3^2 - 1 \geq 0,\,  x_4^2-1 \geq 0.
\earay\right.
\]
The matrix polynomial $L(x) = \bbm \half \diag(x) & -I_4 \ebm$.
%
%
The first part of the objective is $x^Tx$,
while the second part is a nonnegative polynomial \cite{Rez00}.
The objective is coercive, so $f_{\min}$ is achieved at a critical point.
In computation, we got $f_{\min} = 4.0000$ and $11$ global minimizers:
\[
(1, 1, 1, 1), \quad
(1, -1, -1, 1), \quad
(1, -1, 1, -1), \quad
(1, 1, -1, -1),
\]
\[
(1, -1, -1, -1), \quad
(-1, -1, 1, 1), \quad
(-1, 1, -1, 1), \quad
(-1, 1, 1, -1),
\]
\[
(-1, -1, -1, 1), \quad
(-1, -1, 1, -1), \quad
(-1, 1, -1, -1).
\]
The computational results for standard Lasserre's relaxations
and the new ones \reff{momrlx:k}-\reff{rlxsos:k} are in Table~\ref{tab:emp:06}.
It confirms that $f_k = f_{\min}$ for all $k\geq 4$,
up to numerical round-off errors.
\begin{table}[htb]
\centering
\caption{Computational results for Example~\ref{exmp:06}.}
\label{tab:emp:06}
\begin{tabular}{|c|c|r|c|r|} \hline
\multirow{2}{*}{order $k$} &  \multicolumn{2}{c|}{w./o. L.M.E.}
         &  \multicolumn{2}{c|}{with L.M.E.}   \\ \cline{2-5}
  &  lower bound &  time \,\, &  lower bound &  time \,\,  \\ \hline
3 &  $-\infty$      & $1.1377$    & $3.5480$   & $1.1765$    \\   \hline
4 &  $-6.6913 \cdot 10^4$  & $4.7677$    & $4.0000$   & $3.0761$    \\   \hline
5 &  $-21.3778$     & $22.9970$   & $4.0000$   & $10.3354$    \\   \hline
\end{tabular}
\end{table}

\end{exm}

\begin{exm} \label{exmp:07}
Consider the optimization problem
\[
\left\{\baray{rl}
\min\limits_{x\in\re^3} & x_1^4x_2^2 + x_2^4x_3^2 + x_3^4x_1^2 - 3 x_1^2x_2^2 x_3^2 + x_2^2 \\
\mbox{s.t.} &  x_1-x_2x_3 \geq 0, -x_2+x_3^2 \geq 0.
\earay\right.
\]
The matrix polynomial
$
L(x) = \left[\begin{array}{rrrrr} 1 & 0 & 0 & 0 & 0\\ - x_{3} & -1 & 0 & 0 & 0 \end{array}\right].
$
By the arithmetic-geometric mean inequality,
one can show that $f_{\min} =  0 $.
The global minimizers are $(x_1,0,x_3)$ with $x_1 \geq 0$ and $x_1x_3 =0$.
The computational results for standard Lasserre's relaxations
and the new ones \reff{momrlx:k}-\reff{rlxsos:k} are in Table~\ref{tab:emp:07}.
It confirms that $f_k = f_{\min}$ for all $k\geq 5$,
up to numerical round-off errors.
\begin{table}[htb]
\centering
\caption{Computational results for Example~\ref{exmp:07}.}
\label{tab:emp:07}
\begin{tabular}{|c|c|c|c|c|} \hline
\multirow{2}{*}{order $k$} &  \multicolumn{2}{c|}{w./o. L.M.E.}
         &  \multicolumn{2}{c|}{with L.M.E.}   \\ \cline{2-5}
  &  lower bound &  time  &  lower bound &  time   \\ \hline
3 &  $-\infty$    & $0.6144$     & $-\infty$     & $0.3418$    \\   \hline
4 &  $-1.0909 \cdot 10^7$      & $1.0542$     & $-3.9476$     & $0.7180$    \\   \hline
5 &  $-942.6772$  & $1.6771$     & $-3 \cdot 10^{-9}$     & $1.4607$    \\   \hline
6 &  $-0.0110$    & $3.3532$     & $-8 \cdot 10^{-10}$    & $3.1618$    \\   \hline
\end{tabular}
\end{table}

\end{exm}

\begin{exm}\label{exmp:08}
Consider the optimization problem
\[
\left\{\baray{rl}
\min\limits_{x\in\re^4} &
 x_1^2(x_1-x_4)^2 + x_2^2(x_2-x_4)^2 + x_3^2(x_3-x_4)^2 + \\
  &  \qquad 2x_1x_2x_3(x_1+x_2+x_3-2x_4) + (x_1-1)^2   + (x_2-1)^2  + (x_3-1)^2   \\
\mbox{s.t.} &  x_1 - x_2 \geq 0, \,\, x_2 - x_3 \geq 0.
\earay \right.
\]
The matrix polynomial
$
L(x) = \left[\begin{array}{rrrrrr} 1 & 0 & 0 & 0 & 0 & 0\\
1 & 1 & 0 & 0 & 0 & 0 \end{array}\right].
$
In the objective, the sum of the first $4$ terms
is a nonnegative form \cite{Rez00},
while the sum of the last $3$ terms is a coercive polynomial.
The objective is coercive, so $f_{\min}$ is achieved at a critical point.
In computation, we got $f_{\min} \approx 0.9413 $
and a minimizer
\[ (0.5632, \,   0.5632, \,  0.5632, \,   0.7510).  \]
The computational results for standard Lasserre's relaxations
and the new ones \reff{momrlx:k}-\reff{rlxsos:k} are in Table~\ref{tab:emp:08}.
It confirms that $f_k = f_{\min}$ for all $k\geq 3$,
up to numerical round-off errors.
\begin{table}[htb]
\centering
\caption{Computational results for Example~\ref{exmp:08}.}
\label{tab:emp:08}
\begin{tabular}{|c|c|c|c|c|} \hline
\multirow{2}{*}{order $k$} &  \multicolumn{2}{c|}{w./o. L.M.E.}
         &  \multicolumn{2}{c|}{with L.M.E.}   \\ \cline{2-5}
  &  lower bound &  time  &  lower bound &  time   \\ \hline
2 &  $-\infty$       &  $0.3984$    & $-0.3360$ \,\,   & $0.9321$    \\   \hline
3 &  $-\infty$       &  $0.7634$    & $0.9413$     & $0.5240$    \\   \hline
4 &  $-6.4896 \cdot 10^5$  &  $4.5496$    & $0.9413$     & $1.7192$    \\   \hline
5 &  $-3.1645 \cdot 10^3$    & $24.3665$    & $0.9413$     & $8.1228$    \\   \hline
\end{tabular}
\end{table}

\end{exm}

\begin{exm}\label{exmp:09}
Consider the optimization problem
\[
\left\{ \baray{rl}
\min\limits_{x\in\re^4} &  (x_1+ x_2 + x_3 + x_4 + 1)^2
          -4(x_1x_2 +x_2x_3 +x_3x_4 + x_4 + x_1) \\
\mbox{s.t.} &   0 \leq  x_1, \ldots, x_4 \leq  1.
\earay \right.
\]
The matrix $L(x)$ is given in Example~\ref{exm:box}.
The objective is the dehomogenization of Horn's form \cite{Rez00}.
The feasible set is compact, so $f_{\min}$
is achieved at a critical point.
The condition~ii) of Theorem~\ref{thm:pf:tight}
is satisfied.\footnote{
Note that $4-{\sum}_{i=1}^4 x_i^2 = {\sum}_{i=1}^4
\Big( x_i (1-x_i)^2 + (1-x_i) (1+x_i^2)  \Big)
\in \mbox{IQ}(c_{eq}, c_{in}).$
}
The minimum value $f_{\min}=0$.
For each $t\in [0,1]$, the point $(t,0,0,1-t)$ is a global minimizer.
The computational results for standard Lasserre's relaxations
and the new ones \reff{momrlx:k}-\reff{rlxsos:k} are in Table~\ref{tab:emp:09}.
\begin{table}[htb]
\centering
\caption{Computational results for Example~\ref{exmp:09}.}
\label{tab:emp:09}
\begin{tabular}{|c|c|c|c|c|} \hline
\multirow{2}{*}{order} &  \multicolumn{2}{c|}{w./o. L.M.E.}
         &  \multicolumn{2}{c|}{with L.M.E.}   \\ \cline{2-5}
  &  lower bound &  time  &  lower bound &  time   \\ \hline
2 &  $-0.0279$     &  $0.2262$    &  $-5 \cdot 10^{-6}$    &  $1.1835$   \\   \hline
3 &  $-0.0005$     &  $0.4691$    &  $-6 \cdot 10^{-7}$    &  $1.6566$   \\   \hline
4 &  $-0.0001$     &  $3.1098$    &  $-2 \cdot 10^{-7}$    &  $5.5234$   \\   \hline
5 &  $-4 \cdot 10^{-5}$   & $16.5092$     & $-6 \cdot 10^{-7}$ & $19.7320$   \\   \hline
\end{tabular}
\end{table}
\end{exm}

For some polynomial optimization problems,
the standard Lasserre relaxations might converge fast,
e.g., the lowest order relaxation may often be tight.
For such cases, the new relaxations~\reff{momrlx:k}-\reff{rlxsos:k}
have the same convergence property,
but might take more computational time.
The following is such a comparison.

\begin{exm} \label{exmp:10}
Consider the optimization problem ($x_0 :=1$)
\[
\left\{ \baray{rl}
\min\limits_{x\in\re^n} &  \sum_{0 \leq i \le j \le j \le n}
c_{ijk} x_i x_j x_k \\
\mbox{s.t.} &   0 \leq  x\leq  1,
\earay \right.
\]
where each coefficient $c_{ijk}$ is randomly generated (by {\tt randn} in MATLAB).
The matrix $L(x)$ is the same as in Example~\ref{exm:box}.
Since the feasible set is compact, we always have $f_c = f_{\min}$.
The condition~ii) of Theorem~\ref{thm:pf:tight} is satisfied,
because of box constraints.
For this kind of randomly generated problems,
standard Lasserre's relaxations are often tight for the order $k=2$,
which is also the case for the new relaxations \reff{momrlx:k}-\reff{rlxsos:k}.
Here, we compare the computational time that is consumed
by standard Lasserre's relaxations and \reff{momrlx:k}-\reff{rlxsos:k}.
The time is shown (in seconds) in Table~\ref{tab:emp:10}.
For each $n$ in the table, we generate $10$ random instances
and we show the average of the consumed time.
For all instances, standard Lasserre's relaxations
and the new ones~\reff{momrlx:k}-\reff{rlxsos:k}
are tight for the order $k=2$,
while their time is a bit different.
\begin{table}[htb]
\centering
\caption{Consumed time (in seconds) for Example~\ref{exmp:10}.}
\label{tab:emp:10}
\begin{tabular}{|r|c|c|c|c|c|c|c|c|} \hline
$n$ \qquad \,  &9     & 10    & 11   &12     &13     &14   \\   \hline
w./o.~L.M.E.   &1.2569&2.5619 &6.3085&15.8722&35.1675&78.4111 \\   \hline
with\,~L.M.E.  &1.9714&3.8288 &8.2519&20.0310&37.6373&82.4778 \\   \hline
\end{tabular}
\end{table}
We can observe that \reff{momrlx:k}-\reff{rlxsos:k}
consume slightly more time.
\end{exm}

\section{Discussions}
\label{sc:dis}

\subsection{Tight relaxations using preorderings}
\label{ssc:prod}

When the global minimum value $f_{\min}$ is achieved at a critical point,
the problem \reff{pop:c(x)>=0} is equivalent to \reff{g:kkt:opt}.
We proposed relaxations \reff{momrlx:k}-\reff{rlxsos:k}
for solving \reff{g:kkt:opt}. Note that
\[
\mbox{IQ}(c_{eq},c_{in})_{2k} + \mbox{IQ}(\phi,\psi)_{2k} =
\mbox{Ideal}(c_{eq},\phi)_{2k} + \mbox{Qmod}(c_{in},\psi)_{2k}.
\]
If we replace the quadratic module $\mbox{Qmod}(c_{in},\psi)$ by
the preordering of $(c_{in},\psi)$ \cite{LasBok,Lau},
we can get further tighter relaxations.
For convenience, write $(c_{in},\psi)$ as a single tuple
$
(g_1, \ldots, g_\ell).
$
Its preordering is the set
\[
\mbox{Preord}(c_{in},\psi) \, := \,
\sum_{ r_1, \ldots, r_\ell \in \{0, 1\}  }
g_1^{r_1}\cdots g_\ell^{r_\ell}  \Sig[x].
\]
The truncation $\mbox{Preord}(c_{in},\psi)_{2k}$
is similarly defined like $\mbox{Qmod}(c_{in},\psi)_{2k}$
in Section~\ref{sc:pre}.
A tighter relaxation than \reff{momrlx:k}, of the same order $k$, is
\be \label{kmom:pord}
\left\{\baray{rl}
f_k^{\prm, pre}  := \min & \langle f, y \rangle \\
 s.t. & \langle 1, y \rangle = 1, \, L_{c_{eq}}^{(k)}(y) = 0, \, L_{\phi}^{(k)}(y) = 0,  \\
      &  L_{ g_1^{r_1}\cdots g_\ell^{r_\ell} }^{(k)}(y) \succeq 0 \quad
      \forall \, r_1,\ldots, r_\ell \in \{0, 1\}, \\
      &  y \in \re^{ \N_{2k}^n }.
\earay \right.
\ee
Similar to \reff{rlxsos:k}, the dual optimization problem of the above is
\be \label{ksos:pre}
\left\{\baray{rl}
f_k^{pre} := \, \max & \gamma  \\
 s.t. &  f-\gamma \in \mbox{Ideal}(c_{eq},\phi)_{2k} +
             \mbox{Preord}(c_{in},\psi)_{2k}.
\earay \right.
\ee
An attractive property of the relaxations
\reff{kmom:pord}-\reff{ksos:pre} is that:
the conclusion of Theorem~\ref{thm:pf:tight}
still holds, even if none of the conditions i)-iii) there is satisfied.
This gives the following theorem.

\begin{theorem}  \label{thm:tight:pre}
Suppose $\mc{K}_c \ne \emptyset$ and Assumption~\ref{ass:lmd:p} holds.
Then, \[ f_k^{pre}  = f_k^{\prm,pre}  = f_c \]
for all $k$ sufficiently large.
Therefore, if the minimum value $f_{\min}$ of \reff{pop:c(x)>=0}
is achieved at a critical point,
then $f_k^{pre}  = f_k^{\prm,pre}  = f_{\min}$ for all $k$ big enough.
\end{theorem}
\begin{proof}
The proof is very similar to the {\bf Case III} of Theorem~\ref{thm:pf:tight}.
Follow the same argument there. Without Assumption~\ref{ass:sig},
we still  have $\hat{f}(x) \equiv 0$ on the set
\[
\mc{K}_3 \, := \, \{ x \in \re^n \, \mid \,
c_{eq}(x) = 0, \, \phi(x) = 0, \, c_{in}(x) \geq 0, \, \psi(x) \geq 0 \}.
\]
By the Positivstellensatz,
there exists an integer $\ell >0$ and
$q \in \mbox{Preord}(c_{in},\psi)$ such that
$
\hat{f}^{2\ell} + q  \in \mbox{Ideal}(c_{eq},\phi).
$
The resting proof is the same.
\end{proof}

\subsection{Singular constraining polynomials}

As shown in Proposition~\ref{pro:L(x)W(x)=Idt},
if the tuple $c$ of constraining polynomials is nonsingular,
then there exists a matrix polynomial $L(x)$ such that
$
L(x) C(x) = I_m.
$
Hence, the Lagrange multiplier $\lmd$
can be expressed as in \reff{lmd=L1*gf(x)}.
However, if $c$ is not nonsingular, then such $L(x)$ does not exist.
For such cases, how can we express $\lmd$ in terms of $x$
for critical pairs $(x,\lmd)$?
This question is mostly open, to the best of the author's knowledge.

\subsection{Degree bound for $L(x)$}

For a nonsingular tuple $c$ of constraining polynomials,
what is a good degree bound for $L(x)$
in Proposition~\ref{pro:L(x)W(x)=Idt}?
When $c$ is linear, a degree bound is given
in Proposition~\ref{pr:H:Lx}. However, for nonlinear $c$,
an explicit degree bound is not known.
Theoretically, we can get a degree bound for $L(x)$.
In the proof of Proposition~\ref{pro:L(x)W(x)=Idt},
the Hilbert's Nullstellensatz is used for $t$ times.
There exists sharp degree bounds for Hilbert's Nullstellensatz~\cite{Kollar88}.
For each time of its usage, if the degree bound in \cite{Kollar88}
is used, then a degree bound for $L(x)$
can be eventually obtained. However, such obtained bound is enormous,
because the one in \cite{Kollar88} is already exponential
in the number of variables.
An interesting future work is to get
a useful degree bound for $L(x)$.

\subsection{Rational representation of Lagrange multipliers}

In \reff{C(x)*lmd=gf(x)}, the Lagrange multiplier vector
$\lmd$ is determined by a linear equation.
Naturally, one can get
\[
\lmd \, = \, \Big( C(x)^T C(x) \Big)^{-1} C(x)^T \bbm \nabla f(x) \\ 0 \ebm,
\]
when $C(x)$ has full column rank.
This rational representation is expensive for usage,
because its denominator
is typically a high degree polynomial.
However, $\lmd$ might have
rational representations other than the above.
Can we find a rational representation
whose denominator and numerator have low degrees?
If this is possible, the methods for optimizing rational
functions \cite{BHL16,JdeK06,NDG08} can be applied.
This is an interesting question for future research.

\bigskip \noindent
{\bf Acknowledgement} \,
The author was partially supported by the NSF grants
DMS-1417985 and DMS-1619973.
He would like to thank the anonymous referees 
for valuable comments on improving the paper.

\end{document}